\documentclass[12pt]{article}

\usepackage{graphics,amsmath,amssymb,amsthm,mathrsfs}

\newcommand{\br}{\mathbb{R}}

\newcommand{\varep}{\varepsilon}
\newcommand{\brd}{\mathbb{R}^d}

\newtheorem{thm}{Theorem}[section]
\newtheorem{lemma}[thm]{Lemma}

\newtheorem{remark}[thm]{Remark}

\numberwithin{equation}{section}

\begin{document}

\bibliographystyle{amsplain}

\title{Estimates of Eigenvalues and Eigenfunctions \\ in Periodic Homogenization}

\author{Carlos E. Kenig\thanks{Supported in part by NSF grant DMS-0968472}
 \and Fanghua Lin \thanks{Supported in part by NSF grant DMS-1159313}
\and Zhongwei Shen\thanks{Supported in part by NSF grant DMS-1161154}}

\date{ }

\maketitle

\begin{abstract}
For a family
of elliptic operators with rapidly oscillating periodic coefficients,
we study the convergence rates for Dirichlet eigenvalues  and bounds of the normal derivatives of Dirichlet
eigenfunctions.
The results rely on an $O(\varep)$ estimate in $H^1$ for solutions with Dirichlet condition.
\end{abstract}

\section{Introduction}

This paper concerns with the asymptotic behavior of Dirichlet eigenvalues and eigenfunctions for
a family of elliptic operators
with rapidly oscillating coefficients.
More precisely, consider
\begin{equation}\label{operator}
\mathcal{L}_\varep =-\text{div} \big( A\left({x}/{\varep}\right)\nabla\big)
=-\frac{\partial }{\partial x_i}
\left[ a_{ij}^{\alpha\beta}\left(\frac{x}{\varep}\right)
\frac{\partial}{\partial x_j} \right], \quad \varep>0
\end{equation}
(the summation convention is used throughout the paper).
We will assume that $A(y)=(a_{ij}^{\alpha\beta} (y))$
with $1\le i,j\le d$ and $1\le \alpha,\beta\le m$ is real and satisfies the
ellipticity condition
\begin{equation}\label{ellipticity}
\kappa |\xi|^2 \le a_{ij}^{\alpha\beta} (y) \xi_i^\alpha\xi_j^\beta
\le {\kappa^{-1}} |\xi|^2 \ \ 
\text{ for } y\in \br^d \text{ and } \xi=(\xi_i^\alpha)\in \br^{dm},
\end{equation}
where $\kappa\in (0,1)$, and the periodicity condition
\begin{equation}\label{periodicity}
A(y+z)=A(y) \quad \text{ for } y\in \br^d \text{ and } z\in \mathbb{Z}^d.
\end{equation}
The symmetry condition $A^*=A$, i.e., $a_{ij}^{\alpha\beta}=a_{ji}^{\beta\alpha}$,
will also be needed for our main results.
Let $\{\lambda_{\varep, k}\}$ denote the sequence of Dirichlet eigenvalues in an increasing order
for $\mathcal{L}_\varep$ in a bounded domain $\Omega$.
We shall use $\{\lambda_{0,k}\}$ to denote the sequence of Dirichlet eigenvalues in an increasing
order for the homogenized (effective) operator $\mathcal{L}_0$ in $\Omega$.
It is well known that for each $k$ fixed, $\lambda_{\varep, k}\to \lambda_{0, k}$, as $\varep \to 0$.
We are interested in the bounds of $|\lambda_{\varep, k}-\lambda_{0, k}|$, which exhibit explicitly dependence on
$\varep$ and $k$.
The following is one of the main results of the paper.

\begin{thm}\label{theorem-A}
Suppose that $A$ satisfies conditions (\ref{ellipticity})-(\ref{periodicity}) and $A^*=A$.
If $m\ge 2$, we also assume that $A$ is H\"older continuous.
Let $\Omega$ be a bounded $C^{1,1}$ domain (or convex domain in the case $m=1$)
in $\br^d$, $d\ge 2$.
Then
\begin{equation}\label{main-estimate-A}
|\lambda_{\varep, k}-\lambda_{0,k} |\le C\varep\,  (\lambda_{0, k})^{3/2},
\end{equation}
where $C$ is independent of $\varep$ and $k$.
\end{thm}

\begin{remark}\label{remark-1.1}
{\rm
By the mini-max principle and Weyl asymptotic formula,
\begin{equation}\label{weyl}
\lambda_{\varep, k}\approx \lambda_{0,k} \approx k^{\frac{2}{dm}}.
\end{equation}
In view of (\ref{main-estimate-A}) and (\ref{weyl}) we obtain
\begin{equation}\label{1.1-1}
|\lambda_{\varep, k}-\lambda_{0, k}|\le C \varep \,k^{\frac{3}{dm}},
\end{equation}
where $C$ is independent of $\varep$ and $k$.
It also follows from (\ref{weyl}) that  the estimate (\ref{main-estimate-A}) is trivial if $\varep (\lambda_{0,k})^{1/2}\ge 1$.
}
\end{remark}

{\rm
Asymptotic behavior of spectra of the operators
$\{\mathcal{L}_\varep\}$
 is an important problem in periodic homogenization;
results related to the convergence of eigenvalues may be found in
\cite{Kesavan-1}
\cite{Kesavan-2}
\cite{Santosa-Vogelius}
\cite{Jikov-1994}
\cite{Moskow-Vogelius-2}
\cite{Moskow-Vogelius-1}
\cite{Castro-Zuazua-2000}
\cite{Castro-2000}
\cite{KLS2}
\cite{Prange-2011}
(also see recent papers
\cite{Bonder-2012-2} 
\cite{Bonder-2012}
\cite{Bonder-2012-1}
for quasilinear elliptic equations).
In particular,  the estimate $|\lambda_{\varep, k}-\lambda_{0, k}|\le C_k\,  \varep$,
which is known under the assumptions on $A$ and $\Omega$ in Theorem \ref{theorem-A},
may be deduced from the $L^2$ convergence estimate:
$\| u_\varep -u_0\|_{L^2(\Omega)}
\le C \varep \| f\|_{L^2(\Omega)}$, where
$u_\varep$ ($\varep\ge 0$) denotes the solution of the Dirichlet problem:
$\mathcal{L}_\varep (u_\varep)=f$ in $\Omega$ and
$u_\varep=0$ on $\partial\Omega$.
Such $L^2$ estimate, which may be found in \cite{Griso-2006} \cite{KLS2} \cite{kls3} \cite{Suslina-2012}
for smooth domains,
in fact implies that
\begin{equation}\label{1.01}
|\lambda_{\varep, k}-\lambda_{0, k}|\le C\,\varep\, \lambda_{0,k}^2,
\end{equation}
where $C$ is independent of $\varep$ and $k$.
In the case that $\Omega$ is a bounded Lipschitz domain,
it was proved in \cite{KLS2} that $\| u_\varep -u_0\|_{L^2(\Omega)}\le C_\sigma\, \varep (|\ln\varep|+1)^{\frac12 +\sigma}
\|f\|_{L^2(\Omega)}$ for any $\sigma>0$, provided
$A$ satisfies (\ref{ellipticity})-(\ref{periodicity}), $A^*=A$, and $A$
is H\"older continuous.
As a result we obtain 
$$
|\lambda_{\varep, k}-\lambda_{0,k}|\le C_\sigma\, \varep \, (|\ln (\varep)|+1)^{\frac12 +\sigma}
(\lambda_{0,k})^2,
$$
where $C_\sigma$ depends on $\sigma$, but not on $\varep$ or $k$.
}

Our estimate in Theorem \ref{theorem-A} improves the estimate (\ref{1.01})
by a factor of $(\lambda_{0,k})^{1/2}$.
This is achieved by utilizing the following $O(\varep)$ estimate in $H^1_0(\Omega;\br^m)$:
\begin{equation}\label{1.02}
\big\| u_\varep -u_0 -\big\{ \Phi^\beta_{\varep, j} -P_j^\beta\big\} \frac{\partial u_0^\beta}{\partial x_j}
\big\|_{H^1_0(\Omega)}
\le C  \varep \, \| f  \|_{L^2(\Omega)},
\end{equation}
where $C$ depends only on $A$ and $\Omega$.
Here $P_j^\beta(x)=x_j (0, \dots, 1,\dots, 0)$ with $1$ in the $\beta^{th}$ position;
$\Phi_\varep (x)=\big(\Phi_{\varep, j}^\beta(x)\big)$ denotes the so-called matrix of Dirichlet correctors, defined
by
\begin{equation}\label{definition-of-Phi}
\left\{
\aligned
\mathcal{L}_\varep (\Phi_{\varep, j}^\beta) & =0 \quad\  \ \text{ in } \Omega,\\
\Phi_{\varep, j}^\beta  & =P_j^\beta \quad \text{ on } \partial\Omega.
\endaligned
\right.
\end{equation}
We remark that (\ref{1.02}) is a special case of convergence estimates in $W^{1,p}_0(\Omega)$
established in \cite{kls3} for $1<p<\infty$, under the assumption that
$A$ satisfies (\ref{ellipticity})-(\ref{periodicity}) and 
is H\"older continuous.
We provide a direct proof, which also covers the scalar case 
$m=1$ without the smoothness condition, in Section 2.
The proof of Theorem \ref{theorem-A}, which uses (\ref{1.02})
and a minimax argument, is given in Section 3.

In this paper we also study the upper and lower bounds of the normal derivatives of the eigenfunctions for $\mathcal{L}_\varep$.
Let $\phi$ be an eigenfunction of the Dirichlet Laplacian on a Lipschitz domain $\Omega$; i.e., $\phi\in H_0^1(\Omega)$ and
$-\Delta \phi=\lambda \phi$ in $\Omega$. Assume that $\|\phi\|_{L^2(\Omega)}=1$.
It follows from the Rellich identity that
\begin{equation}\label{Laplace-estimate}
\int_{\partial\Omega} \big|\frac{\partial\phi}{\partial n}\big|^2\, d\sigma \le C \lambda,
\end{equation}
where $C$ depends only on $\Omega$.
The argument works equally well for second-oder elliptic operators with Lipschitz continuous coefficients.
In fact it was proved in \cite{Hassell-Tao-2002}
that the estimate (\ref{Laplace-estimate}) holds
if $\Omega$ is a general smooth compact Riemannian manifold with boundary. Furthermore, 
the lower bound $c\lambda \le \|{\partial \phi}/{\partial n}\|^2_{L^2(\partial\Omega)}$
 holds, if $\Omega$ has no trapped geodesics
(see related work in \cite{Ozawa-1993} \cite{Xu-2011}; we were kindly informed by N. Burq that
 the results on upper and lower bounds in \cite{Hassell-Tao-2002} may be deduced from 
 earlier work on the wave equations in \cite{Bardos-1992} \cite{Burq-1997}).

A very interesting problem is whether the estimate (\ref{Laplace-estimate}) holds
for eigenfunctions of $\mathcal{L}_\varep$, with constant $C$ {\it independent} of $\varep$
and $\lambda$.
This problem is closely related to the uniform boundary controllability of the wave operator
$\frac{\partial^2}{\partial t^2}+\mathcal{L}_\varep$ (see e.g. \cite{lions-1988} 
\cite{AL-1989-ho} \cite{Avellaneda-1992} \cite{Castro-2000}
\cite{Lebeau-2000}  and their references).
In the case $m=d=1$, it is known that the estimate (\ref{Laplace-estimate}) with constant $C$ independent of
$\varep$ and $\lambda$ may fail. Counter-examples of eigenfunctions $\phi_\varep $
 with eigenvalues $\lambda_\varep \sim \varep^{-2}$
can be constructed so that
\begin{equation}\label{counter-example}
\int_{\partial\Omega} \big|\frac{\partial \phi_\varep}{\partial n}\big |^2\, d\sigma\sim (\lambda_\varep)^{3/2}
\end{equation}
(see e.g. \cite{Castro-2000}). We remark that asymptotic behavior of
eigenvalues and eigenfunctions below and above the critical size $(\lambda_{\varep, k} \sim \varep^{-2})$
was investigated rather extensively for $d=m=1$ in \cite{Castro-1999} \cite{Castro-Zuazua-2000} \cite{Castro-2000}.
To the best of our knowledge, the only results for the case $d\ge 2$ were contained in \cite{Lebeau-2000},
where an observability estimate for a wave equation with rapidly oscillating density was established.
Note that if $d=1$, equations with oscillating coefficients are equivalent to those with 
oscillating potentials. This, however, is not the case in higher dimensions.

In this paper we show that the estimate (\ref{Laplace-estimate}) holds if $\varep \lambda_\varep\le 1$.
In fact  we obtain the following.

\begin{thm}\label{theorem-B}
Suppose that $A$ satisfies (\ref{ellipticity})-(\ref{periodicity}) and $A^*=A$.
Also assume that $A$ is Lipschitz continuous.
Let $\Omega$ be a bounded $C^{1,1}$ domain in $\br^d$, $d\ge 2$.
Let $\phi_\varep\in H_0^1(\Omega;\br^m) $ be a Dirichlet eigenfunction for $\mathcal{L}_\varep$ in $\Omega$
with the associated eigenvalue 
$\lambda_\varep $ and $\|\phi_\varep \|_{L^2(\Omega)}=1$.
Then
\begin{equation}\label{main-estimate-B}
\int_{\partial\Omega}
|\nabla \phi_\varep |^2\, d\sigma
\le\left\{
\begin{array}{ll}
C\lambda_\varep (1+\varep^{-1}) \quad & \text{ if } \ \varep^2 \lambda_\varep \ge 1,\\
 C\lambda_\varep(1+ \varep \lambda_\varep) \quad & \text{ if } \ \varep^2 \lambda_\varep <1,
\end{array}
\right.
\end{equation}
where $C$ depends only on $A$ and $\Omega$.
\end{thm}

If $\varep \lambda_\varep $ is sufficiently small, we also obtain a sharp lower bound in the case of scalar equations.

\begin{thm}\label{theorem-C}
Let $m=1$ and $\Omega$ be a bounded $C^2$ domain in $\br^d$, $d\ge 2$.
Suppose that $A$ satisfies the same conditions as in Theorem \ref{theorem-B}.
Let $\phi_\varep\in H^1_0(\Omega)$ be a Dirichlet eigenfunction with the associated eigenvalue $\lambda_\varep$ and 
$\|\phi_\varep\|_{L^2(\Omega)} =1$. Then there exists $\delta>0$ such that if $\lambda_\varep>1$ and
$\varep \lambda_\varep <\delta $,
\begin{equation}\label{main-estimate-C}
\int_{\partial\Omega} |\nabla \phi_\varep|^2\, d\sigma \ge c\, \lambda_\varep,
\end{equation}
where $\delta>0$ and $c>0$ depend only on $A$ and $\Omega$.
\end{thm}

\begin{remark}\label{remark-1.3}
{\rm
It follows from (\ref{main-estimate-B}) that
\begin{equation}\label{1.3-1}
\int_{\partial\Omega} |\nabla \phi_\varep|^2\, d\sigma
\le C\, (\lambda_\varep)^{3/2},
\end{equation}
where $C$ depends only on $A$ and $\Omega$.
In Section 4 we provide a direct proof of (\ref{1.3-1}), 
under the weaker assumptions that $\Omega$ is Lipschitz, $A$ satisfies (\ref{ellipticity})-(\ref{periodicity}),
$A^*=A$, and $A$ is H\"older continuous.
The proof uses the $L^2$ Rellich estimates established  in \cite{Kenig-Shen-2}.
}
\end{remark}

Let $\{\phi_{\varep, k}\}$ be an orthonormal basis of $L^2(\Omega; \br^m)$, where
$\phi_{\varep, k}$ is a Dirichlet eigenfunction for $\mathcal{L}_\varep$ in $\Omega$ with eigenvalue $\lambda_{\varep, k}$.
The spectral (cluster) projection operator $S_{\varep,\lambda} (f)$ is defined by
\begin{equation}\label{definition-of-S}
S_{\varep,\lambda} (f)
=\sum_{\sqrt{\lambda_{\varep, k}}\in 
\big[\sqrt{\lambda}, \sqrt{\lambda} +1 \big)}
\phi_{\varep, k} (f),
\end{equation}
where $\lambda\ge 1$, $ \phi_{\varep, k}(f) (x) =<\phi_{\varep, k}, f> \phi_{\varep, k} (x)$, and
$<\, , \, >$ denotes the inner product in $L^2(\Omega; \br^m)$.
Let $u_\varep=S_{\varep, \lambda} (f)$, where $f\in L^2(\Omega; \br^m)$ 
and $\| f\|_{L^2(\Omega)}=1$. We will show in Section 4 that
\begin{equation}\label{1.05}
\int_{\partial\Omega}
|\nabla u_\varep|^2\, d\sigma
\le 
\left\{
\begin{array}{ll}
 C\, \lambda (1+\varep^{-1}) \  & \text{ if }\  \varep^2 {\lambda}\ge 1,\\
 C \, \lambda  ( 1+ \varep \lambda ) \  &\text{ if } \ \varep^2 {\lambda}<1,
\end{array}
\right.
\end{equation}
where $C$ depends only on $A$ and $\Omega$.
Theorem \ref{theorem-B} follows if we choose $f$ to be an eigenfunction of $\mathcal{L}_\varep$.
We point out that while the estimate in (\ref{1.05}) for the case $\varep^2{\lambda}\ge 1$,
 as in the case of Laplacian \cite{Xu-2011},  follows readily from the Rellich
identities,
the proof for the case $\varep^2 {\lambda}<1$ is more subtle.
The basic idea is to use the $H^1$ convergence estimate (\ref{1.02}) to approximate the eigenfunction $\phi_\varep$
with eigenvalue $\lambda_\varep$ by the solution $v_\varep$ of the Dirichlet problem:
$\mathcal{L}_0 (v_\varep)=\lambda_\varep \phi_\varep$ in $\Omega$ and $v_\varep =0$ in $\partial\Omega$.
The same approach, together with a compactness argument,
 also leads to the sharp lower bound in Theorem \ref{theorem-C}, whose proof is given in Section 5.

\section{Convergence rates in $H^1$}

Let $\mathcal{L}_\varep=-\text{\rm div} (A(x/\varep)\nabla )$ 
with $A(y)=\big(a_{ij}^{\alpha\beta}(y)\big)$ satisfying
(\ref{ellipticity})-(\ref{periodicity}).
Let $\chi(y)=\big(\chi_j^{\alpha\beta} (y)\big)$ denote the matrix of correctors 
for $\mathcal{L}_1$ in $\br^d$, where
 $\chi_j^\beta (y)
=\big(\chi_j^{1\beta}(y), \dots, \chi_j^{m\beta}(y)\big)\in H_{\rm per}^1(Y; \br^m)$ 
is defined by the following cell problem:
\begin{equation}\label{cell-problem}
\left\{
\aligned
& \mathcal{L}_1 (\chi_j^\beta)=-\mathcal{L}_1(P_j^\beta)\quad \text{ in } \br^d,\\
&\chi_j^{\beta} \text{ is periodic with respect to }\mathbb{Z}^d
\text{ and } \int_Y
\chi_j^{\beta} \, dy =0, 
\endaligned
\right.
\end{equation}
for each $1\le j\le d$ and $1\le \beta\le m$.
Here $Y=[0,1)^d\simeq \brd/\mathbb{Z}^d$ and $P_j^\beta (y)
=y_j (0,\dots, 1, \dots, 0)$ with $1$ in the $\beta^{th}$ position.
The homogenized operator is given by
 $\mathcal{L}_0=-\text{div}(\widehat{A}\nabla)$, where 
$\widehat{A} =(\hat{a}_{ij}^{\alpha\beta})$ and
\begin{equation}
\label{homogenized-coefficient}
\hat{a}_{ij}^{\alpha\beta}
=\int_Y
\left[ a_{ij}^{\alpha\beta}
+a_{ik}^{\alpha\gamma}
\frac{\partial}{\partial y_k}\left( \chi_j^{\gamma\beta}\right)\right]
\, dy.
\end{equation}
Let
\begin{equation}\label{definition-of-B}
b_{ij}^{\alpha\beta} (y)
=\hat{a}_{ij}^{\alpha\beta}
-a_{ij}^{\alpha\beta} (y)
-a_{ik}^{\alpha\gamma} (y) \frac{\partial}{\partial y_k} \big(\chi_j^{\gamma\beta}\big),
\end{equation}
where $1\le \alpha, \beta\le m$ and $1\le i, j\le d$.

\begin{lemma}\label{lemma-2.0}
Suppose that $A$ satisfies conditions (\ref{ellipticity})-(\ref{periodicity}). 
For $1\le \alpha, \beta\le m$ and $1\le i,j,k\le d$, 
there exists
$F_{kij}^{\alpha\beta}\in H_{\rm per}^1(Y)$ such that
\begin{equation}\label{definition-of-F}
b_{ij}^{\alpha\beta} =\frac{\partial}{\partial y_k}
\big\{ F_{kij}^{\alpha\beta}\big\}
\quad \text{ and } \quad
F_{kij}^{\alpha\beta}=-F_{ikj}^{\alpha\beta}.
\end{equation}
Moreover,  $F=(F^{\alpha\beta}_{kij}) \in L^\infty(Y)$ if $\chi=(\chi_j^{\alpha\beta})$ is H\"older continuous.
\end{lemma}

\begin{proof}
See Remark 2.1 in \cite{kls3}.
\end{proof}

By the N. Meyer estimates (see e.g. \cite[p.154]{Giaquinta}), the matrix of correctors $\chi\in W_{\rm per} ^{1,p}(Y)$ for some $p>2$.
It follows that $\chi$ is H\"older continuous if $d=2$.
In the scalar case $(m=1)$, the well known De Giorgi -Nash estimates
also give the H\"older continuity of $\chi$ for $d\ge 3$.
In view of Lemma \ref{lemma-2.0} we may deduce that $\| F^{\alpha\beta}_{kij}\|_\infty\le C$ if
$d=2$ and $m\ge 1$, or $d\ge 3$ and $m=1$,
where $C$ depends only on $d$ and $\kappa$.
If $d\ge 3$ and $m\ge 2$, the functions $F^{\alpha\beta}_{kij}$ (and $\nabla F^{\alpha\beta}_{kij}$)
 are bounded if $A$ is H\"older continuous.

\begin{lemma}\label{lemma-2.1}
Suppose that $A$ satisfies conditions (\ref{ellipticity})-(\ref{periodicity}).
Let $m=1$ and $\Omega$ be a bounded Lipschitz domain.
Then
\begin{equation}\label{estimate-of-Phi}
\| \Phi_{\varep, j}^\beta-P_j^\beta \|_{L^\infty(\Omega)}
\le C \varep,
\end{equation}
where $C$ depends only on $A$.
If $m\ge 2$, the estimate (\ref{estimate-of-Phi}) holds, with $C$ depending only on 
$A$ and $\Omega$, under the additional assumptions that
 $A$ is H\"older continuous
and $\Omega$ is $C^{1,\alpha}$ for some $\alpha\in (0,1)$.
\end{lemma}

\begin{proof}
This is proved in \cite[Proposition 2.4]{kls3} by considering the function
$u_\varep =\Phi_{\varep, j}^\beta (x) -P_j^\beta (x)-\varep \chi_j^\beta(x/\varep)$.
Notice that $\mathcal{L}(u_\varep)=0$ in $\Omega$ and
$u_\varep =-\varep\chi^\beta_j(x/\varep)$ on $\partial\Omega$.
In the scalar case one may use the maximum principle and boundedness of
$\chi$ to show that
$\|u_\varep\|_{L^\infty(\Omega)}\le \|u_\varep\|_{L^\infty(\partial\Omega)}
\le C \varep$.
This implies that $\|\Phi_{\varep, j}^\beta-P_j^\beta\|_{L^\infty(\Omega)}\le C \varep$.
 If $m\ge 2$, under the additional assumptions that $A$ is H\"older continuous  and
$\Omega$ is $C^{1, \alpha}$, we know that $\chi$ is bounded and
 $\|u_\varep\|_{L^\infty(\Omega)}
\le C\,  \| u_\varep\|_{L^\infty(\partial\Omega)}$ (see \cite[p.805, Theorem 3]{AL-1987}).
This again gives (\ref{estimate-of-Phi}).
\end{proof}

\begin{lemma}\label{lemma-2.2}
Suppose that $u_\varep\in H^1(\Omega;\br^m), u_0\in H^2(\Omega;\br^m)$, and $\mathcal{L}_\varep (u_\varep)
=\mathcal{L}_0(u_0)$ in $\Omega$.
Let
\begin{equation}\label{definition-of-w}
w_\varep (x) =u_\varep (x)-u_0 (x) -\big\{ \Phi_{\varep, j}^\beta (x) -P_j^\beta (x)\big\}
\cdot \frac{\partial u_0^\beta}{\partial x_j}.
\end{equation}
Then
\begin{equation}\label{formula-2.1}
\aligned
\big(\mathcal{L}_\varep (w_\varep)\big)^\alpha
=& \varep \frac{\partial}{\partial x_i}
\left\{ \left[ F_{jik}^{\alpha\gamma} \left({x}/{\varep}\right)
\right]
\frac{\partial^2 u_0^\gamma}{\partial x_j\partial x_k}\right\}\\
&+\frac{\partial}{\partial x_i}
\left\{ a_{ij}^{\alpha\beta}\left({x}/{\varep}\right)
\left[ \Phi_{\varep, k}^{\beta\gamma}(x) 
-x_k \delta^{\beta\gamma}\right]
\frac{\partial^2 u_0^\gamma}{\partial x_j\partial x_k}\right\}\\
&
+a_{ij}^{\alpha\beta} \left({x}/{\varep}\right)
\frac{\partial}{\partial x_j}
\left[ \Phi_{\varep, k}^{\beta\gamma}(x)
-x_k \delta^{\beta\gamma}
-\varep \chi_k^{\beta\gamma}\left({x}/{\varep}\right) \right]
\frac{\partial^2 u_0^\gamma}{\partial x_i\partial x_k},
\endaligned
\end{equation}
where $\delta^{\beta\gamma}=1$ if $\beta=\gamma$, and zero otherwise.
\end{lemma}

\begin{proof}
This follows from Proposition 2.2 in \cite{kls3} by taking $V_{\varep, j}^\beta (x)=\Phi_{\varep, j}^\beta (x)$.
\end{proof}

\begin{thm}\label{H-1-theorem}
Suppose that $A$ satisfies (\ref{ellipticity})-(\ref{periodicity}).
If $m\ge 2$,  assume further that $A$ is H\"older continuous.
Let $\Omega$ be a $C^{1,1}$ domain in $\br^d$.
For $\varep\ge 0$ and $f\in L^2(\Omega;\br^m)$,
 let $u_\varep$ be the unique weak solution in $H^1_0(\Omega;\br^m)$ to the elliptic
system $\mathcal{L}_\varep (u_\varep)=f$ in $\Omega$.
Then
\begin{equation}\label{H-1-estimate}
\big\| u_\varep -u_0 -\big\{ \Phi^\beta_{\varep, j} -P_j^\beta\big\} \frac{\partial u_0^\beta}{\partial x_j}
\big\|_{H^1_0(\Omega)}
\le C  \varep \, \| f  \|_{L^2(\Omega)},
\end{equation}
where $C$ depends only on $A$ and $\Omega$.
\end{thm}

\begin{proof}
Under the assumption that $A$ satisfies (\ref{ellipticity})-(\ref{periodicity}) and is H\"older continuous,
the estimate (\ref{H-1-estimate}) is a special case of the convergence  estimates 
in $W^{1,p}_0(\Omega;\br^m)$ for $1<p<\infty$,
 proved in 
\cite[Theorem 3.7]{kls3}.
We give a direct proof here, which covers the case $m=1$ without the smoothness condition.

Let $w_\varep$ be given by (\ref{definition-of-w}).
We first consider the case $f\in C_0^\infty(\br^d;\br^m)$.
In this case it is easy to see that under the assumptions in the theorem,
$w_\varep\in H_0^1(\Omega;\br^m)\cap L^\infty(\Omega;\br^m)$.
It follows from (\ref{formula-2.1}) that
\begin{equation}\label{2.4-1}
\aligned
\kappa\int_\Omega |\nabla w_\varep|^2\, dx
 \le C \varep \int_\Omega &  |\nabla^2 u_0|\, |\nabla w_\varep|\, dx\\
&+C \int_\Omega | \nabla \big\{
\Phi_\varep (x) -P(x) -\varep \chi(x/\varep)\big\}|\, 
|\nabla^2 u_0|\, |w_\varep|\, dx,
\endaligned
\end{equation}
where $\Phi_\varep =\big( \Phi_{\varep, j}^\beta\big)$,
$P=\big(P_j^\beta\big)$, and
we have used estimates $\|F_{kij}^{\alpha\beta}\|_\infty\le C$
in Lemma \ref{lemma-2.0} and $\|\Phi_\varep-P\|_\infty\le C \varep$ in Lemma \ref{lemma-2.1}.
By the Cauchy inequality this implies that
\begin{equation}\label{2.4-3}
\aligned
\int_\Omega
|\nabla w_\varep|^2\, dx
\le \frac{C \varep^2}{\delta} & \int_\Omega |\nabla^2 u_0|^2\, dx\\
& +\frac{\delta}{\varep^2}
\int_\Omega |\nabla \big\{ \Phi_\varep (x) -P (x) -\varep\chi (x/\varep)\big\}|^2\, |w_\varep|^2\, dx
\endaligned
\end{equation}
for any $\delta\in (0,1)$.
We claim that
\begin{equation}\label{2.4-5}
\int_\Omega
 |\nabla \big\{ \Phi_\varep (x) -P (x) -\varep\chi (x/\varep)\big\}|^2\, |w_\varep|^2 \, dx
\le C_0 \varep^2 \int_\Omega |\nabla w_\varep|^2\, dx.
\end{equation}
By choosing $\delta>0$ so small that $C_0 \delta<(1/2)$, 
we may deduce from (\ref{2.4-3}) and (\ref{2.4-5}) that
$$
\| w_\varep\|_{H^1_0(\Omega)}
\le C\, \|\nabla w_\varep\|_{L^2(\Omega)} \le C \varep \, \| \nabla^2 u_0\|_{L^2(\Omega)}
\le C \varep \, \|f\|_{L^2(\Omega)}.
$$

To see (\ref{2.4-5}), we fix $1\le \beta_0\le m$ and $1\le j_0\le d$ and let
$$
h_\varep (x)=\Phi_{\varep, j_0}^{\beta_0}(x) -P_{j_0}^{\beta_0} (x) -\varep \chi_{j_0}^{\beta_0} (x/\varep) \quad \text{ in } 
\Omega.
$$
Note that $h_\varep\in H^1(\Omega; \br^m)\cap L^\infty(\Omega;\br^m)$ and
$\mathcal{L}_\varep (h_\varep)=0$ in $\Omega$. It follows that
\begin{equation}\label{2.4-7}
\aligned
\kappa \int_\Omega |\nabla h_\varep|^2\, |w_\varep|^2\, dx
& \le \int_\Omega
a_{ij}^{\alpha\beta}(x/\varep) \frac{\partial h_\varep^\alpha}{\partial x_i}
\cdot \frac{\partial h_\varep^\beta}{\partial x_j}\, |w_\varep|^2\, dx\\
& =-2\int_\Omega
h_\varep^\alpha \cdot a_{ij}^{\alpha\beta} (x/\varep) \frac{\partial h_\varep^\beta}{\partial x_j}\cdot
\frac{\partial w_\varep^\gamma}{\partial x_i} \, w_\varep^\gamma\, dx.
\endaligned
\end{equation}
Hence,
\begin{equation}\label{2.4-9}
\int_\Omega |\nabla h_\varep|^2 |w_\varep|^2\, dx
\le C \int_\Omega |h_\varep|\, |\nabla h_\varep| \, |\nabla w_\varep|\, |w_\varep|\, dx,
\end{equation}
where $C$ depends only on $d$ and $\kappa$.
Estimate (\ref{2.4-5}) now follows from (\ref{2.4-9})
by the Cauchy inequality and
the fact that $\|h_\varep\|_\infty \le C \varep$.

Finally, suppose $f\in L^2(\Omega;\br^m)$.
Choose a sequence of functions $\{ f_\ell\}$ in $C_0^\infty (\Omega;\br^m)$ such that
$f_\ell \to f$ in $L^2(\Omega;\br^m)$.
Let $w_{\varep, \ell}$ be defined by (\ref{definition-of-w}), but with $f$ replaced by $f_\ell$.
Since $$
\| w_{\varep, j} -w_{\varep, \ell } \|_{H^1_0(\Omega)}\le C \varep \, \| f_j -f_\ell\|_{L^2(\Omega)},
$$
it follows that $w_{\varep, \ell} \to \widetilde{w}$ in $H_0^1(\Omega;\br^m)$
as $\ell \to \infty$, and $\|\widetilde{w}\|_{H^1_0(\Omega)} \le C \varep\| f\|_{L^2(\Omega)}$.
However, it is not hard to verify that $w_{\varep, \ell} \to w_\varep$ in $L^2(\Omega;\br^m)$.
As a result we may conclude that $w_\varep =\widetilde{w}\in H_0^1(\Omega;\br^m)$
and the estimate (\ref{H-1-estimate}) holds.
This completes the proof.
\end{proof}

\begin{remark}\label{remark-2.1}
{\rm
Let $m=1$ and $\Omega$ be a bounded Lipschitz domain.
An inspection of the proof of Theorem \ref{H-1-theorem} shows that
 the estimate (\ref{H-1-estimate}) continues to hold
as long as one has $\|\nabla^2 u_0\|_{L^2(\Omega)}
\le C\, \| f\|_{L^2(\Omega)}$ and $\nabla u_0\in L^\infty(\Omega;\br^m)$ for $f\in C_0^\infty(\Omega; \br^m)$.
Consequently, the estimate  (\ref{H-1-estimate}) holds in the scalar case, if $\Omega$ is convex
and $A$ satisfies (\ref{ellipticity}) and (\ref{periodicity}).
}
\end{remark}

\begin{remark}\label{remark-2.2}
{\rm
Since
$$
\aligned
\frac{\partial}{\partial x_i}
\bigg\{ u_\varep -u_0 &  -\left\{ \Phi_{\varep, j}^\beta -P_j^\beta\right\}
 \frac{\partial u_0^\beta}{\partial x_j} \bigg\}
\\
& =\frac{\partial u_\varep}{\partial x_i}
-\frac{\partial}{\partial x_i} \left\{ \Phi_{\varep, j}^\beta \right\}\cdot \frac{\partial u_0^\beta}{\partial x_j}-
\left\{ \Phi_{\varep, j}^\beta -P_j^\beta\right\} 
\frac{\partial^2 u_0^\beta}{\partial x_i\partial x_j},
\endaligned
$$
it follows from (\ref{H-1-estimate}) and (\ref{estimate-of-Phi}) as well as the estimate $\|\nabla^2 u_0\|_{L^2(\Omega)}
\le C\, \| f\|_{L^2(\Omega)}$ that
\begin{equation}\label{estimate-2.5}
\big\| 
\frac{\partial u_\varep}{\partial x_i}
-\frac{\partial}{\partial x_i} \left\{ \Phi_{\varep, j}^\beta \right\}\cdot \frac{\partial u_0^\beta}{\partial x_j}\big\|_{L^2(\Omega)}
\le C\varep \| f\|_{L^2(\Omega)}.
\end{equation}
}
\end{remark}

%
%
%
%
%
%

\section{Convergence rates for eigenvalues}

The goal of this section is to prove Theorem \ref{theorem-A}.
For $\varep\ge 0$ and $f\in L^2(\Omega; \br^m)$, under
conditions (\ref{ellipticity}) and (\ref{periodicity}),
the elliptic 
system $\mathcal{L}_\varep (u_\varep) =f$ in $\Omega$
has a unique (weak) solution in $H^1_0(\Omega; \br^m)$ .
Define $T_\varep (f)=u_\varep$. 
Since $\|u_\varep\|_{H^1_0(\Omega)} \le C\, \| f\|_{L^2(\Omega)}$, where $C$ depends only on
$\kappa$ and $\Omega$, the linear operator 
 $T_\varep$ is bounded, positive, and compact  on $L^2(\Omega; \br^m)$.
Under the symmetry condition $A^*=A$, the operator $T_\varep$ is also self-adjoint.
Let 
\begin{equation}\label{eigenvalue}
\mu_{\varep, 1}\ge \mu_{\varep, 2}\ge \cdots \ge \mu_{\varep, k} \ge \cdots >0
\end{equation}
be the sequence of eigenvalues, in a decreasing order, of $T_\varep$.
By the mini-max principle,
\begin{equation}\label{mini-max}
\mu_{\varep, k}
=\min_{\substack{f_1, \cdots, f_{k-1}\\ \in L^2(\Omega; \br^m)}}\
\max_{\substack{ \| f\|_{L^2(\Omega)}=1\\ f\perp f_i\\ i=1, \dots, k-1}}
 <T_\varep (f), f>,
\end{equation}
where $<\, , \, >$ denotes the inner product in $L^2(\Omega; \br^m)$.
Note that
\begin{equation}\label{bi-linear-form}
<T_\varep(f), f>
=<u_\varep, f>
=\int_\Omega a_{ij}^{\alpha\beta}(x/\varep) \frac{\partial u^\alpha}{\partial x_i} \cdot \frac{\partial u^\beta}{\partial x_j}\, dx
\end{equation}
(if $\varep =0$, $a_{ij}^{\alpha\beta}(x/\varep)$ is replaced by $\hat{a}_{ij}^{\alpha\beta}$).

Let $\{ \phi_{\varep, k} \}$ be an orthonormal basis of $L^2(\Omega; \br^m)$, where $\phi_{\varep, k}$ is 
an eigenfunctions associated with $\mu_{\varep, k}$.
Let $V_{\varep, 0}=\{ 0\}$ and $V_{\varep, k}$ be the subspace of $L^2(\Omega; \br^m)$
 spanned by $\{\phi_{\varep, 1}, \dots, \phi_{\varep, k}\}$ for $k\ge 1$.
Then
\begin{equation}\label{3.1-1}
\mu_{\varep, k}=\max_{\substack{f\perp V_{\varep, k-1}\\ \| f\|_{L^2(\Omega)}=1}}
<T_\varep (f), f>.
\end{equation}
Let $\lambda_{\varep, k}= (\mu_{\varep, k})^{-1}$.
Then $\{\lambda_{\varep, k}\}$ is the sequence of Dirichlet eigenvalues in an increasing order
of $\mathcal{L}_\varep$ in $\Omega$.

\begin{lemma}\label{lemma-3.1}.
Suppose that $A$ satisfies (\ref{ellipticity})-(\ref{periodicity}) and the symmetry condition
$A^*=A$.
Then
$$
|\mu_{\varep, k} -\mu_{0, k}|
\le \max \left\{
\max_{\substack{f\perp V_{0, k-1} \\ \| f\|_{L^2(\Omega)}=1}}
|<(T_\varep-T_0 )f, f>|,\, 
\max_{\substack{f\perp V_{\varep, k-1} \\ \| f\|_{L^2(\Omega)}=1}}
|<(T_\varep-T_0 )f, f>|\right\}
$$
for any $\varep>0$.
\end{lemma}

\begin{proof}
It follows from (\ref{mini-max}) that
$$
\aligned
\mu_{\varep, k} & \le \max_{\substack{f\perp V_{0, k-1}\\ \| f\|_{L^2(\Omega)}=1}}
<T_\varep (f), f>\\
&\le \max_{\substack{f\perp V_{0, k-1}\\ \| f\|_{L^2(\Omega)}=1}}
< (T_\varep-T_0)( f), f>
+
\max_{\substack{f\perp V_{0, k-1}\\ \| f\|_{L^2(\Omega)}=1}}
< T_0 (f), f>\\
& =\max_{\substack{f\perp V_{0, k-1}\\ \| f\|_{L^2(\Omega)}=1}}
< (T_\varep-T_0) (f), f>
+\mu_{0,k},
\endaligned
$$
where we have used (\ref{3.1-1}). Hence,
\begin{equation}\label{3.1-3}
\mu_{\varep, k}-\mu_{0, k}
\le \max_{\substack{f\perp V_{0, k-1}\\ \| f\|_{L^2(\Omega)}=1}}
< (T_\varep-T_0) (f), f>.
\end{equation}
Similarly, one can show that
\begin{equation}\label{3.1-5}
\mu_{0, k}-\mu_{\varep, k}
\le \max_{\substack{f\perp V_{\varep, k-1}\\ \| f\|_{L^2(\Omega)}=1}}
< (T_0-T_\varep) (f), f>.
\end{equation}
The desired estimate follows readily from (\ref{3.1-3}) and (\ref{3.1-5}).
\end{proof}

It follows from Lemma \ref{lemma-3.1} that
\begin{equation}\label{3.2-1}
|\mu_{\varep, k}-\mu_{0,k}|\le \| T_\varep -T_0\|_{L^2\to L^2}.
\end{equation}
Under the assumptions in Theorem \ref{theorem-A}, it is known that 
$\| u_\varep -u_0\|_{L^2(\Omega)}\le C\varep \| f\|_{L^2(\Omega)}$, where $C$ depends on
$A$ and $\Omega$. Hence
$\|T_\varep -T_0\|_{L^2\to L^2} \le C \varep$, which
 implies that $|\mu_{\varep, k}-\mu_{\varep, 0}|\le C\varep$.
It follows that
$$
|\lambda_{\varep, k}-\lambda_{0, k} |\le C \varep \lambda_{0,k}\lambda_{\varep, k}.
$$
By the mini-max principle and Weyl's asymptotic,
 $\lambda_{\varep, k}\approx \lambda_{0,k}\approx k^{\frac{2}{dm}}$.
 As a result, we obtain
 \begin{equation}\label{3.2-3}
 |\lambda_{\varep, k}-\lambda_{0, k}|\le C \varep (\lambda_{0, k})^2 \le C \varep k^{\frac{4}{dm}},
 \end{equation}
 where $C$ is independent of $\varep $ and $k$. 
 Note that the proof of (\ref{3.2-3}) relies on the convergence estimate in $L^2$:
 $\| u_\varep -u_0\|_{L^2(\Omega)} \le C \varep \| f\|_{L^2(\Omega)}$.
 The convergence estimate in $H_0^1$ in Theorem \ref{H-1-theorem}
 allows us to improve the estimate (\ref{3.2-3}) by a factor of $k^{1/(dm)}$.

\begin{proof}[\bf Proof of Theorem \ref{theorem-A}]
We will use Lemma \ref{lemma-3.1} and Theorem \ref{H-1-theorem} to show that
\begin{equation}\label{3.3-1}
|\mu_{\varep, k}-\mu_{0, k}|\le C \, \varep\,  (\mu_{0, k})^{1/2},
\end{equation}
where $C$ is independent of $\varep $ and $k$.
Since $\lambda_{\varep, k}=(\mu_{\varep, k})^{-1}$ for $\varep\ge 0$ and
$\lambda_{\varep, k}\approx \lambda_{0, k}$,
this gives the desired estimate.

Let $u_\varep =T_\varep (f)$ and
$u_0=T_0(f)$, where $\| f\|_{L^2(\Omega)}=1$ and $f\perp V_{0,k-1}$.
In view of (\ref{3.1-1}) for $\varep=0$, we have $<u_0, f>\le \mu_{0, k}$.
Hence,
$$
c\|\nabla u_0\|^2_{L^2(\Omega)} \le <u_0, f> \le \mu_{0, k},
$$
where $c>0$ depends only on the ellipticity constant $\kappa$ of $A$. It follows that
\begin{equation}\label{3.3-3}
\| f\|_{H^{-1}(\Omega)}
\le C \, \|\nabla u_0\|_{L^2(\Omega)} \le C\, (\mu_{0, k})^{1/2}.
\end{equation}
Now, write
$$
<u_\varep-u_0, f>
=\big<u_\varep -u_0 - \big\{ \Phi_{\varep, \ell}^\beta -P_\ell^\beta\big\} \frac{\partial u_0^\beta}{\partial x_\ell}, f\big>
+\big<\big\{ \Phi_{\varep, \ell}^\beta -P_\ell^\beta\big\} \frac{\partial u_0^\beta}{\partial x_\ell}, f\big>.
$$
This implies that for any $f\perp V_{0, k-1}$ with $\| f\|_{L^2(\Omega)}=1$,
\begin{equation}\label{3.3-4}
\aligned
|<u_\varep-u_0,f>|
&\le  \big\|u_\varep-u_0-
\big\{ \Phi_{\varep, \ell}^\beta -P_\ell^\beta\big\} \frac{\partial u_0^\beta}{\partial x_\ell}\big\|_{H^1_0(\Omega)}
\| f\|_{H^{-1}(\Omega)}\\
&\qquad\qquad 
+\big\| \big\{ \Phi_{\varep, \ell}^\beta -P_\ell^\beta\big\} \frac{\partial u_0^\beta}{\partial x_\ell}\big\|_{L^2(\Omega)}
\| f\|_{L^2(\Omega)}\\
& \le C \varep \| f\|_{L^2(\Omega)} \| f\|_{H^{-1}(\Omega)}
+ C\varep \|\nabla u_0\|_{L^2(\Omega)} \| f\|_{L^2(\Omega)}\\
&\le C\varep \| \nabla u_0\|_{L^2(\Omega)}\\
& \le C \varep (\mu_{0,k})^{1/2},
\endaligned
\end{equation}
where we have used Theorem \ref{H-1-theorem} and the estimate 
$\|\Phi_{\varep, \ell}^\beta-P_\ell^\beta\|_\infty \le C\varep$
for  the second inequality, and (\ref{3.3-3}) for the third and fourth.

Next we consider the case $f\perp V_{\varep, k-1}$ and $\| f\|_{L^2(\Omega)}=1$.
In view of (\ref{3.1-1}) we have $<u_\varep, f>\le \mu_{\varep, k}$.
Hence, $c\|\nabla u_\varep\|_{L^2(\Omega)}^2 \le <u_\varep, f>\le \mu_{\varep, k}$.
It follows that
\begin{equation}\label{3.3-5}
\| f\|_{H^{-1}(\Omega)}
\le C \|\nabla u_\varep\|_{L^2(\Omega)} \le C (\mu_{\varep, k})^{1/2}
\end{equation}
and
\begin{equation}\label{3.3-7}
\|\nabla u_0\|_{L^2(\Omega)}
\le C \, \| f\|_{H^{-1}(\Omega)}
\le C \, (\mu_{\varep, k})^{1/2},
\end{equation}
where $C$ depends only on the ellipticity constant of $A$.
As before, this implies that
for any $f\perp V_{\varep, k-1}$ with $\|f\|_{L^2(\Omega)}=1$,
\begin{equation}\label{3.3-6}
\aligned
|<u_\varep-u_0,f>|
&\le  \big\|u_\varep-u_0-
\big\{ \Phi_{\varep, \ell}^\beta -P_\ell^\beta\big\} \frac{\partial u_0^\beta}{\partial x_\ell}\big\|_{H^1_0(\Omega)}
\| f\|_{H^{-1}(\Omega)}\\
&\qquad\qquad 
+\big\| \big\{ \Phi_{\varep, \ell}^\beta -P_\ell^\beta\big\} \frac{\partial u_0^\beta}{\partial x_\ell}\big\|_{L^2(\Omega)}
\| f\|_{L^2(\Omega)}\\
& \le C \varep \| f\|_{H^{-1}(\Omega)}+
C \varep \|\nabla u_0\|_{L^2(\Omega)}\\
&\le C \varep (\mu_{\varep, k})^{1/2}\\
&\le C \varep (\mu_{0, k})^{1/2},
\endaligned
\end{equation}
where we have used the fact $\mu_{\varep, k}\approx \mu_{0, k}$.
In view of Lemma \ref{lemma-3.1}, the estimate (\ref{3.3-1}) follows from (\ref{3.3-4}) and (\ref{3.3-6}).
\end{proof}


\section{Conormal derivatives of Dirichlet eigenfunctions}

Throughout this section we assume that $A$ satisfies conditions (\ref{ellipticity})-(\ref{periodicity})
and $A^*=A$.
Let $\lambda\ge 1$ and $S_{\varep, \lambda} (f)$ be defined by (\ref{definition-of-S}).
Note that
\begin{equation}\label{4.1.1}
\mathcal{L}_\varep \big( S_{\varep, \lambda} (f)\big)
=\lambda S_{\varep, \lambda} (f)+ R_{\varep, \lambda} (f),
\end{equation}
where
\begin{equation}\label{4.1.3}
R_{\varep, \lambda} (f)(x)
=\sum_{\sqrt{\lambda_{\varep, k}}\in 
\big[\sqrt{\lambda}, \sqrt{\lambda} +1 \big)}
(\lambda_{\varep, k}-\lambda) \phi_{\varep, k} (f).
\end{equation}
Clearly, $\|S_{\varep, \lambda} (f)\|_{L^2(\Omega)}\le \| f\|_{L^2(\Omega)}$.
It is also not hard to see that
\begin{equation}\label{4.1.5}
\aligned
\|\nabla S_{\varep, \lambda} (f)\|_{L^2(\Omega)} & \le C \sqrt{\lambda}\, \| f\|_{L^2(\Omega)},\\
\|R_{\varep, \lambda} (f)\|_{L^2(\Omega)}
 & \le C  \sqrt{\lambda}\,  \| f\|_{L^2(\Omega)},\\
 \|\nabla R_{\varep, \lambda} (f)\|_{L^2(\Omega)}
&  \le C  {\lambda} \,  \| f\|_{L^2(\Omega)},
 \endaligned
 \end{equation}
 where $C$ depends only on the ellipticity constant $\kappa$ of $A$.
 
 \begin{lemma}\label{Rellich-lemma}
 Suppose that $A$ satisfies (\ref{ellipticity})-(\ref{periodicity}) and $A^*=A$.
 Also assume that $A$ is Lipschitz continuous.
 Let $u_\varep\in H^2(\Omega;\br^m)$ be a solution of 
 $\mathcal{L}_\varep (u_\varep)=f$ in $\Omega$ for some
 $f\in L^2(\Omega;\br^m)$, where $\Omega$ is a bounded Lipschitz
 domain. Then
 \begin{equation}\label{Rellich-identity}
 \aligned
 \int_{\partial\Omega}
 n_kh_k \, a_{ij}^{\alpha\beta}(x/\varep) \frac{\partial u_\varep^\alpha}{\partial x_i}
 \cdot \frac{\partial u_\varep^\beta}{\partial x_j}\, d\sigma
=2 & \int_{\partial\Omega}   h_k
\left\{n_k \frac{\partial }{\partial x_i} -n_i \frac{\partial}{\partial x_k}\right\} u_\varep^\alpha
\cdot a_{ij}^{\alpha\beta} (x/\varep)\, 
\frac{\partial u_\varep ^\beta}{\partial x_j}\, d\sigma\\
&-\int_\Omega \text{\rm div} (h) \, 
a_{ij}^{\alpha\beta}(x/\varep) \frac{\partial u_\varep^\alpha}{\partial x_i}
 \cdot \frac{\partial u_\varep ^\beta}{\partial x_j}\, dx\\
&  -\int_\Omega
 h_k \frac{\partial} {\partial x_k} \big\{ a_{ij}^{\alpha\beta} (x/\varep)\big\} 
 \frac{\partial u_\varep ^\alpha}{\partial x_i}
 \cdot \frac{\partial u_\varep^\beta}{\partial x_j}\, dx\\
&+2\int_\Omega \frac{\partial h_k}{\partial x_i} \cdot a_{ij}^{\alpha\beta}(x/\varep)
\frac{\partial u_\varep^\alpha}{\partial x_k}
 \cdot \frac{\partial u_\varep ^\beta}{\partial x_j}\, dx\\
&
-2 \int_\Omega f^\alpha \cdot \frac{\partial u_\varep^\alpha}{\partial x_k} \cdot h_k \, dx,
\endaligned
\end{equation}
where $h=(h_1, \dots, h_m)\in C_0^1(\br^d; \br^d)$ and $n$ denotes the unit outward normal to $\partial\Omega$.
\end{lemma}

\begin{proof}
Use the divergence theorem and the assumption that $A^*=A$.
We refer the reader to \cite{fabes2} for the
case of constant coefficients.
\end{proof}

\begin{lemma}\label{lemma-4.2}
Assume that $A$ and $\Omega$ satisfy the same assumptions as in Lemma \ref{Rellich-lemma}.
Let $u_\varep =S_{\varep, \lambda} (f)$ be defined by (\ref{definition-of-S}),
 where $f\in L^2(\Omega; \br^m)$ and
$\| f\|_{L^2(\Omega)} =1$.
Suppose that $u_\varep \in H^2(\Omega;\br^m)$.
Then
\begin{equation}\label{estimate-4.2}
\int_{\partial\Omega}
|\nabla u_\varep|^2\, d\sigma
\le C\, {\lambda} 
+\frac{C}{\varep} \int_{\Omega_\varep}
|\nabla u_\varep|^2\, dx,
\end{equation}
where $\Omega_\varep =\big\{ x\in \Omega:\, \text{\rm dist} (x, \partial\Omega)< \varep\big\}$
and $C$ depends only on $A$ and $\Omega$.
\end{lemma}

\begin{proof}
We first consider the case $0<\varep<\text{diam}(\Omega)$.
In this case we may choose a vector field $h$ in $C_0^1(\br^d; \br^d)$ such that
$n_kh_k\ge c>0$ on $\partial\Omega$, $|h|\le 1$,
 $|\nabla h|\le C \varep^{-1}$,  and
$h=0$ on $\{ x\in \Omega: \text{dist}(x, \partial\Omega)\ge c\varep\}$,
where $c=c(\Omega)>0$ is small.
Note that $\mathcal{L}_\varep (u_\varep)=\lambda u_\varep +R_{\varep, \lambda} (f)$ in $\Omega$.
Since $u_\varep =0$ on $\partial\Omega$, it follows from (\ref{Rellich-identity}) that
\begin{equation}\label{4.2-1}
\aligned
c\int_{\partial\Omega}|\nabla u_\varep|^2\, d\sigma
&\le \frac{C}{\varep} \int_{\Omega_\varep} |\nabla u_\varep|^2\, dx
-2\lambda \int_\Omega  u_\varep^\alpha \cdot 
\frac{\partial u_\varep^\alpha}{\partial x_k}
 \cdot  h_k \, dx\\
& \qquad \qquad
-2 \int_\Omega \big(R_{\varep, \lambda} (f)\big)^\alpha \cdot
\frac{\partial u_\varep^\alpha}{\partial x_k}\cdot h_k \, dx.
\endaligned
\end{equation}
Using the Cauchy inequality we may bound
 the third integral in the right hand side of (\ref{4.2-1})
by $C\| R_{\varep, \lambda} (f)\|_{L^2(\Omega)} \|\nabla u_\varep\|_{L^2(\Omega)}$,
which, in view of (\ref{4.1.5}),
 is dominated by $C\lambda$.

To handle the second integral in the right hand side of
(\ref{4.2-1}), we use the integration by parts to obtain
\begin{equation}\label{4.2-3}
\big|2\lambda
 \int_\Omega  u_\varep^\alpha \cdot 
\frac{\partial u_\varep^\alpha}{\partial x_k}
 \cdot  h_k \, dx\big|
   = \big| \lambda \int_\Omega |u_\varep|^2 \, \text{div}(h)\, dx\big|\\
 \le \frac{C\lambda}{\varep} \int_{\Omega_{c\varep}}
|u_\varep|^2\, dx.
\end{equation}
Since
\begin{equation}\label{4.2-5}
\lambda |u_\varep|^2 -a_{ij}^{\alpha\beta} (x/\varep) \frac{\partial u_\varep^\alpha}{\partial x_i}
\cdot \frac{\partial u_\varep^\beta}{\partial x_j}
=\big( \lambda u_\varep -\mathcal{L}_\varep (u_\varep)\big)^\alpha u_\varep^\alpha
-\frac{\partial}{\partial x_i}
\left\{ u_\varep^\alpha \, a_{ij}^{\alpha\beta}(x/\varep)
\frac{\partial u_\varep^\beta}{\partial x_j}\right\},
\end{equation}
it follows that for any $\varphi\in C^1_0(\br^d)$,
\begin{equation}\label{4.2-7}
\aligned
& \int_\Omega
\left\{
 \lambda |u_\varep|^2 -a_{ij}^{\alpha\beta} (x/\varep) \frac{\partial u_\varep^\alpha}{\partial x_i}
\cdot \frac{\partial u_\varep^\beta}{\partial x_j}
\right\} \varphi^2\, dx\\
&
=\int_\Omega 
\big( \lambda u_\varep -\mathcal{L}_\varep (u_\varep)\big)^\alpha u_\varep^\alpha
\varphi^2\, dx
+2\int_\Omega
u^\alpha_\varep
a_{ij}^{\alpha\beta} (x/\varep) \frac{\partial u_\varep^\beta}{\partial x_j} \cdot
\frac{\partial\varphi}{\partial x_i}\, \varphi\ dx.
\endaligned
\end{equation}
Choose $\varphi$ so that $0\le \varphi\le 1$, $\varphi (x)=1$ if dist$(x, \partial\Omega)\le c\varep$,
$\varphi(x)=0$ if dist$(x, \partial\Omega)\ge 2c\varep$,
and $|\nabla\varphi|\le C \varep^{-1}$.
In view of (\ref{4.2-7}) we have
$$
\aligned
\lambda
\int_\Omega |u_\varep|^2 \varphi^2\, dx
&\le C \int_\Omega |\nabla u_\varep|^2 \varphi^2\, dx
+\int_\Omega |R_{\varep, \lambda} (f)|\, |u_\varep |\, \varphi^2\, dx
+C \int_\Omega |u_\varep |^2\, |\nabla \varphi |^2\, dx\\
& 
\le C \int_\Omega |\nabla u_\varep|^2 \varphi^2\, dx
+\|R_{\varep, \lambda} (f)\|_{L^2(\Omega)}  \|u_\varep\|_{L^2(\Omega_{2c\varep})}
+\frac{C}{\varep^2} \int_{\Omega_{2c\varep}} |u_\varep|^2\, dx\\
&\le  C\int_{\Omega_{2c\varep}} |\nabla u_\varep|^2\, dx
+C \varep  {\lambda},
\endaligned
$$
where we have used the Cauchy inequality, (\ref{4.1.5}), and the inequality
\begin{equation}\label{4.2-9}
\int_{\Omega_{2c\varep}} |u_\varep|^2\, dx \le C\varep^2 \int_{\Omega_{2c\varep}}
|\nabla u_\varep|^2\, dx.
\end{equation}
This, together with (\ref{4.2-1}) and (\ref{4.2-3}), gives the estimate (\ref{estimate-4.2}).

Finally, if $\varep\ge \text{diam}(\Omega)$, we choose a vector field $h\in C_0^1(\br^d; \br^d)$ so that
$h_kn_k\ge c>0$ on $\partial\Omega$.
The same argument as in (\ref{4.2-1}) and (\ref{4.2-3})
 shows that the left hand side of (\ref{estimate-4.2}) is bounded by $C\lambda $.
\end{proof}

\begin{thm}\label{theorem-4.1}
Suppose that $A$ satisfies conditions (\ref{ellipticity})-(\ref{periodicity}),
and $A^*=A$.
Also assume that $A$ is Lipschitz continuous.
Let $\Omega$ be a bounded $C^{1,1}$ domain.
Let $u_\varep =S_{\varep, \lambda} (f)$ be defined by (\ref{definition-of-S}),
 where $f\in L^2(\Omega, \br^m)$ and
$\| f\|_{L^2(\Omega)} =1$.
Then
\begin{equation}\label{estimate-4.1}
\int_{\partial\Omega}
|\nabla u_\varep|^2\, d\sigma
\le 
\left\{
\begin{array}{ll}
 C\lambda (1+\varep^{-1}) \  & \text{ if }\  \varep^2 \lambda \ge 1,\\
 C \lambda ( 1+ \varep \lambda ) \  &\text{ if } \ \varep^2\lambda<1,
\end{array}
\right.
\end{equation}
where $C$ depends only on $A$ and $\Omega$.
\end{thm}

\begin{proof}
We first note that under the conditions on $A$ and $\Omega$ in the theorem, 
$u_\varep\in H^2(\Omega; \br^m)$.
This allows us to use Lemma \ref{lemma-4.2} and reduce the problem
to the estimate of $\varep^{-1} \|\nabla u_\varep\|^2_{L^2(\Omega_\varep)}$
by the right hand side of (\ref{estimate-4.1}).
If $\varep^2 \lambda \ge 1$, the desired estimate follows directly from 
$\|\nabla u_\varep\|^2_{L^2(\Omega)} \le C \lambda$.

The proof for the case $\varep^2\lambda<1$ is more subtle and uses
the $H^1$ convergence estimate in Theorem \ref{H-1-theorem}.
Let $v_\varep$ be the unique solution in $ H_0^1(\Omega; \br^m)$ to the system,
\begin{equation}\label{4.5-1}
\mathcal{L}_0(v_\varep)=\lambda u_\varep +R_{\varep, \lambda} (f) \quad \text{ in } \Omega.
\end{equation}
Observe that
\begin{equation}\label{4.5-2}
\|\lambda u_\varep +R_{\varep, \lambda} (f)\|_{L^2(\Omega)}
\le C \lambda .
\end{equation}
Since $\partial\Omega$ is $C^{1,1}$ and $\mathcal{L}_0$ is a second order elliptic operator with constant coefficients, 
this implies that $v_\varep \in H^2(\Omega;\br^m)$ and
\begin{equation}\label{4.5-4}
\|\nabla^2 v_\varep\|_{L^2(\Omega)} \le C \lambda .
\end{equation}
Also, using $\mathcal{L}_0(v_\varep)=\mathcal{L}_\varep(u_\varep)$ in $\Omega$, we may deduce that
\begin{equation}\label{4.5-3}
\| v_\varep\|_{H^1_0(\Omega)}
\le C \| \nabla u_\varep\|_{L^2(\Omega)}
\le C \sqrt{\lambda},
\end{equation}
where we have used (\ref{4.1.5}).
To estimate $\varep^{-1}\| \nabla u_\varep\|^2_{L^2(\Omega_\varep)}$, we use 
the estimate $\|\nabla\Phi_\varep\|_\infty \le C$  in \cite{AL-1987} to obtain
\begin{equation}\label{4.5-5}
\aligned
\frac{1}{\varep}
\int_{\Omega_\varep}
\big|\frac{\partial u_\varep}{\partial x_i}\big|^2\, dx
& \le 
\frac{C}{\varep}\int_{\Omega_\varep} 
\big|\frac{\partial u_\varep}{\partial x_i}
-\frac{\partial}{\partial x_i}
\left\{ \Phi_{\varep, j}^\beta\right\} \cdot \frac{\partial v_\varep^\beta}{\partial x_j}\big|^2\, dx
+\frac{C}{\varep} \int_{\Omega_\varep} |\nabla v_\varep|^2\, dx\\
&\le C\varep \lambda^2 +\frac{C}{\varep}
\int_{\Omega_\varep} |\nabla v_\varep|^2\, dx,
\endaligned
\end{equation}
where the last inequality follows from
(\ref{estimate-2.5}) and (\ref{4.5-2}).
Furthermore, we may use the Fundamental Theorem of Calculus to obtain
$$
\aligned
\frac{1}{\varep}
\int_{\Omega_\varep} |\nabla v_\varep|^2\, dx
& \le C \int_{\partial\Omega} |\nabla v_\varep|^2\, d\sigma
+C \varep \int_{\Omega_\varep} |\nabla^2 v_\varep|^2\, dx\\
&\le C \int_{\partial\Omega} |\nabla v_\varep|^2\, d\sigma
+C \varep \lambda^2,
\endaligned
$$
where we have used (\ref{4.5-4}) for the second inequality.
As a result it suffices to show that
\begin{equation}\label{4.5-6}
\int_{\partial\Omega} |\nabla v_\varep|^2\, d\sigma
\le C \lambda (1+\varep \lambda).
\end{equation}

To this end we use a Rellich identity for $\mathcal{L}_0$, similar to (\ref{Rellich-identity})
for $\mathcal{L}_\varep$, to deduce that
\begin{equation}\label{4.5-7}
\aligned
\int_{\partial\Omega}
|\nabla v_\varep|^2\, dx
& \le C \int_\Omega |\nabla v_\varep|^2\, dx
+C \big|\int_\Omega
\big\{ \lambda u_\varep +R_{\varep, \lambda} (f)\big\}^\alpha \cdot \frac{\partial v_\varep^\alpha}
{\partial x_k}  \cdot h_k\, dx\big|\\
&\le C \lambda
+C \lambda \big|\int_\Omega
 u_\varep^\alpha \cdot \frac{\partial v_\varep^\alpha}
{\partial x_k}  \cdot h_k\, dx\big|\\
&
\le C \lambda
+C \lambda \big|\int_\Omega
 v_\varep^\alpha \cdot \frac{\partial u_\varep^\alpha}
{\partial x_k}  \cdot h_k\, dx\big|
+\lambda \big|\int_\Omega u_\varep^\alpha \cdot v_\varep^\alpha \cdot \text{div}(h)\, dx\big|, 
\endaligned
\end{equation}
where $h=(h_1, \dots, h_d)\in C_0^1(\br^d; \br^d)$ is a vector field such that
$h_kn_k\ge c>0$ on $\partial\Omega$ and $|h|+|\nabla h|\le C$, and
we have used (\ref{4.5-3}) and (\ref{4.1.5})  for the second inequality and
integration by parts for the third.
To estimate the third integral in the right hand side of (\ref{4.5-7}), we note that
$$
\aligned
\| u_\varep\|_{H^{-1}(\Omega)}
& = \lambda^{-1} \| \mathcal{L}_\varep (u_\varep) -R_{\varep, \lambda} (f)\|_{H^{-1}(\Omega)}\\
&\le  C \lambda^{-1} \|\nabla u_\varep\|_{L^2(\Omega)}
+ C\lambda^{-1} \|R_{\varep, \lambda} (f)\|_{L^2(\Omega)}\\
& \le C\lambda^{-1/2},
\endaligned
$$
where we have used (\ref{4.1.5}).
It follows that
\begin{equation}\label{4.5-9}
\lambda \big|\int_\Omega u_\varep^\alpha \, v_\varep^\alpha \, \text{div}(h)\, dx\big|
\le C \lambda\,  \| u_\varep\|_{H^{-1}(\Omega)}
\| v_\varep \, \text{div} (h)\|_{H^1_0(\Omega)}\\
\le C \lambda.
\end{equation}

Finally, we claim that
\begin{equation}\label{claim}
  \big|\int_\Omega
 v_\varep^\alpha \, \frac{\partial u_\varep^\alpha}
{\partial x_k}  \, h_k\, dx\big|
\le C  (1+\varep\lambda).
\end{equation}
In view of (\ref{4.5-7}) and (\ref{4.5-9}), this would give the estimate (\ref{4.5-6}).
To see (\ref{claim}) we use integration by parts to obtain
$$
\aligned
  \big|\int_\Omega
 v_\varep^\alpha \, \frac{\partial u_\varep^\alpha}
{\partial x_k}  \, h_k\, dx\big|
&\le 
  \big|\int_\Omega
(u_\varep^\alpha - v_\varep^\alpha) \, \frac{\partial u_\varep^\alpha}
{\partial x_k}  \, h_k\, dx\big|
+\frac{1}{2}
\big|\int_\Omega |u_\varep|^2\, \text{div}(h)\, dx \big|\\
&\le 
 \big|\int_\Omega
\left\{ u_\varep^\alpha - v_\varep^\alpha-\left\{ \Phi_{\varep, j}^{\alpha\beta} -x_j\delta^{\alpha\beta} \right\}
\frac{\partial v^\beta_\varep}{\partial x_j} \right\} \, \frac{\partial u_\varep^\alpha}
{\partial x_k}  \, h_k\, dx\big| \\
& \qquad\qquad +
 \big|\int_\Omega
\left\{ \Phi_{\varep, j}^{\alpha\beta} -x_j \delta^{\alpha\beta} \right\}
\frac{\partial v^\beta_\varep}{\partial x_j} \cdot \frac{\partial u_\varep^\alpha}
{\partial x_k}  \, h_k\, dx\big|
+C \\
& \le 
C \, \| (\nabla u_\varep) h\|_{H^{-1}(\Omega)}
\big\| u_\varep - v_\varep-\left\{ \Phi_{\varep, j}^\beta -P_j^\beta\right\}
\frac{\partial v^\beta_\varep}{\partial x_j}  
\big\|_{H_0^1(\Omega)}\\
& \qquad\qquad+C \varep\,  \|\nabla u_\varep\|_{L^2(\Omega)}
\|\nabla v_\varep\|_{L^2(\Omega)}
+C \\
&\le C  +C \varep \lambda,
\endaligned
$$
where we have used Theorem \ref{H-1-theorem} as well as the estimate 
$\|\nabla u_\varep\|_{L^2(\Omega)} +\|\nabla v_\varep \|_{L^2(\Omega)} \le C \lambda^{1/2}$
for the last inequality.
This completes the proof.
\end{proof}

Note that the right hand side of (\ref{estimate-4.1}) is bounded
by $C\lambda^{3/2}$ in both cases.
We give a direct proof of this weaker estimate
under some weaker assumptions.

\begin{thm}\label{theorem-4.7}
Assume that $A$ satisfies (\ref{ellipticity})-(\ref{periodicity}), $A^*=A$,
and $A$ is H\"older continuous.
Let $\Omega$ be a bounded Lipschitz domain.
Let $u_\varep= S_{\varep,\lambda} (f)$ be defined as in (\ref{definition-of-S}).
Then
\begin{equation}\label{estimate-4.7}
\int_{\partial\Omega}
|\nabla u_{\varep} |^2\, d\sigma
\le C \lambda ^{3/2}\int_\Omega |f|^2\, dx, 
\end{equation}
where $C$ depends only on $\Omega$ and $A$.
\end{thm}

\begin{remark}
{\rm
Recall from (\ref{counter-example}) that the upper bound (\ref{estimate-4.7}) is sharp when $d=1$ and $\Omega$ is an interval.
}
\end{remark}

The proof of Theorem \ref{theorem-4.7}
relies on the Rellich estimate in the following lemma.

\begin{lemma}\label{lemma-4.9}
Assume that $A$ and $\Omega$ satisfy the same conditions as in Theorem \ref{theorem-4.7}.
Suppose that $u_\varep\in H^1(\Omega; \br^m)$ and
$\mathcal{L}_\varep (u_\varep)=f$ in $\Omega$ for some $f\in L^2(\Omega; \br^m)$.
We further assume that $u_\varep \in H^1(\partial\Omega; \br^m)$.
Then
\begin{equation}\label{estimate-4.9}
\int_{\partial\Omega} |\nabla u_\varep|^2\, d\sigma
\le C \int_{\partial\Omega} 
|\nabla_{\tan} u_\varep|^2\, d\sigma
+C \int_{\partial \Omega} |u_\varep|^2\, d\sigma
+ C\int_\Omega |f|^2\, dx,
\end{equation}
where $\nabla_{\tan} u_\varep$ denotes the tangential gradient of $u_\varep$ on $\partial\Omega$ and
$C$ depends only on $A$ and $\Omega$.
\end{lemma}

\begin{proof}
We first point out  that in the case $f=0$, the estimate (\ref{estimate-4.9})
was proved in \cite{Kenig-Shen-2} for Lipschitz domains with connected boundaries.
If $\partial\Omega$ is not connected, the estimate 
\begin{equation}\label{4.9-1}
\|\nabla u_\varep\|_{L^2(\partial\Omega)}
\le C\,  \| u_\varep \|_{H^1(\partial\Omega)}
\end{equation}
follows from the case of connected boundary by a localization argument.

If $f=(f^1, \dots, f^m)\neq 0$, we define $w_\varep=(w_\varep^1(x), \dots, w_\varep^m(x))$ by
$$
w^\alpha_\varep (x)=\int_\Omega \Gamma_\varep^{\alpha\beta} (x,y) f^\beta (y)\, dy,
$$
where $\Gamma_\varep (x,y)$ is the matrix of fundamental
solutions for $\mathcal{L}_\varep$ in $\br^d$, with pole at $y$.
Then $w_\varep\in H^1(\Omega;\br^m)$ and $\mathcal{L}_\varep (w_\varep)=f$ in $\Omega$.
We claim that
\begin{equation}\label{4.9-3}
\int_{\partial\Omega} |\nabla w_\varep|^2\, d\sigma 
+\int_{\partial\Omega} |w_\varep|^2\, d\sigma
\le C \int_\Omega |f|^2\, dx.
\end{equation}
Assume the claim (\ref{4.9-3}) for a moment.
Note that $u_\varep -w_\varep\in H^1(\Omega)$,
$\mathcal{L}_\varep (u_\varep -w_\varep)=0$ in $\Omega$, and
$u_\varep-w_\varep\in H^1(\partial\Omega)$.
In view of estimate (\ref{4.9-1}) for the case $f=0$, we obtain
$ \|\nabla (u_\varep-w_\varep)\|_{L^2(\partial\Omega)}
\le C\, \| u_\varep-w_\varep\|_{H^1(\partial\Omega)}$.
This, together with (\ref{4.9-3}), yields that
$$
\aligned
\|\nabla u_\varep\|_{L^2(\partial\Omega)}
& \le C\, \|u_\varep -w_\varep\|_{H^1(\partial\Omega)} + \|\nabla w_\varep\|_{L^2(\partial\Omega)}\\
&\le C\,  \|\nabla_{\tan} u_\varep\|_{L^2(\partial\Omega)}
+C\,  \|u_\varep\|_{L^2(\partial\Omega)} + C\, \| \nabla w_\varep\|_{L^2(\partial\Omega)}
+C \| w_\varep\|_{L^2(\partial\Omega)}
\\
&
\le C\, \|\nabla_{\tan} u_\varep\|_{L^2(\partial\Omega)}
+C \, \|u_\varep\|_{L^2(\partial\Omega)}
+ C\, \| f\|_{L^2(\Omega)}.
\endaligned
$$

It remains to prove (\ref{4.9-3}). We will assume that $f\in C_0^1(\Omega; \br^m)$; the general case follows by a limiting 
argument.
Let $g=(g^1, \dots, g^m)\in L^2(\partial\Omega; \br^m)$.
It follows from Fubini's theorem as well as the Cauchy inequality that
\begin{equation}\label{4.9-7}
\aligned
\big|\int_{\partial\Omega} \frac{\partial w_\varep^\alpha}{\partial x_i}\,  g^\alpha\, d\sigma\big|
& =\big|\int_\Omega
f^\beta (y)\left\{\int_{\partial\Omega}
\frac{\partial}{\partial x_i} \left\{ \Gamma_\varep^{\alpha\beta} (x,y)\right\} g^\alpha (x)\, d\sigma(x)\right\}\, dy\big|\\
& \le \| f\|_{L^2(\Omega)} \| v_\varep\|_{L^2(\Omega)},
\endaligned
\end{equation}
where $v_\varep =(v_\varep^1, \dots, v_\varep^m)$ and
$$
v^\beta_\varep (y)=
\int_{\partial\Omega}
\frac{\partial}{\partial x_i}
\left\{ \Gamma_\varep^{\alpha\beta} (x,y)\right\} g^\alpha (x)\, d\sigma(x).
$$
By \cite[Theorem 3.5]{Kenig-Shen-2}, we have
$$
\| v_\varep\|_{L^2(\Omega)} \le C\, \|(v_\varep)^*\|_{L^2(\partial\Omega)} \le C\, \| g\|_{L^2(\partial\Omega)},
$$
where $(v_\varep)^*$ denotes the nontangential maximal function of $v_\varep$.
In view of (\ref{4.9-7}), this, by duality, implies that $\|\nabla w_\varep\|_{L^2(\partial\Omega)}\le C \, \| f\|_{L^2(\Omega)}$.

Finally,  we note that since
$|\Gamma_\varep (x,y)|\le C |x-y|^{2-d}$ (see \cite{AL-1987}),
$$
|w_\varep (x)|\le C \int_{\Omega} \frac{|f(y)|}{|x-y|^{d-2}}\, dy
\le C \left\{ \int_\Omega \frac{|f(y)|^2}{|x-y|^{d-2}}\, dy\right\}^{1/2}.
$$
This yields the estimate $\| w_\varep\|_{L^2(\partial\Omega)}
\le C \| f\|_{L^2(\Omega)}$.
\end{proof}

\begin{proof}[\bf Proof of Theorem \ref{theorem-4.7}]

We may assume that $\|f \|_{L^2(\Omega)}=1$.
Consider the function
\begin{equation}\label{4.10-1}
w_\varep (x,t)=u_\varep (x) \cosh (\sqrt{\lambda}\, t)\qquad \text{ in } \Omega_T
\end{equation}
where $\Omega_T =\Omega \times (0, T)$ and $T=\text{diam}(\Omega)$.
Note that $\Omega_T$ is a bounded Lipschitz domain in $\br^{d+1}$
and $w_\varep\in H^1(\Omega_T)$.
Since $\mathcal{L}_\varep (u_\varep)=\lambda u_\varep +R_{\varep, \lambda} (f)$ in $\Omega$, it 
follows that
$$
\left\{ \mathcal{L}_\varep -\frac{\partial^2}{\partial t^2} \right\} w_\varep =R_{\varep, \lambda} (f) \cosh(\sqrt{\lambda}\, t)
\qquad \text{ in } \Omega_T.
$$
In view of Lemma \ref{lemma-4.9} we obtain
$$
\int_{\partial \Omega_T}
|\nabla_{x,t} w|^2\, d\sigma (x,t)
\le C \int_{\partial\Omega_T} |\nabla_{\tan} w|^2\, d\sigma (x,t)
+C \int_{\Omega_T}
|R_{\varep, \lambda } (f) \, \cosh(\sqrt{\lambda}\, t)|^2\, dxdt.
$$
This implies that
\begin{equation}\label{4.10-3}
\aligned
&\int_0^T |\cosh(\sqrt{\lambda}\, t)|^2\, dt \int_{\partial\Omega}|\nabla u_\varep|^2\, d\sigma\\
&\qquad \le C \lambda |\cosh(\sqrt{\lambda}\, T)|^2 
+C \lambda \int_0^T |\cosh(\sqrt{\lambda}\, t)|^2\, dt,
\endaligned
\end{equation}
where we have used the fact $w_\varep=0$ on $\partial\Omega \times (0,T)$ as well
as estimates of $\|\nabla u_\varep\|_{L^2(\Omega)}$
and $\|R_{\varep, \lambda}(f)\|_{L^2(\Omega)}$ in (\ref{4.1.5}).
Finally, since
$$
\int_0^T |\cosh(\sqrt{\lambda}\, t)|^2\, dt \sim \frac{1}{\sqrt{\lambda}} e^{2\sqrt{\lambda} T}\sim \frac{1}{\sqrt{\lambda}}
|\cosh (\sqrt{\lambda}\, T)|^2,
$$
we may deduce from (\ref{4.10-3}) that
$$
\int_{\partial\Omega}|\nabla u_\varep|^2\, d\sigma \le C \lambda^{3/2}.
$$
This finishes the proof.
\end{proof}



\section{Lower bounds}

In this section we give the proof of Theorem \ref{theorem-C}.
Throughout this section
 we will assume that $ m=1$ and $\Omega$ is a bounded $C^2$ domain in
$\br^d$, $d\ge 2$.
We will also assume that $A$ satisfies (\ref{ellipticity})-(\ref{periodicity}), $A^*=A$, and
$A$ is Lipschitz continuous. 

Recall that $\Phi_\varep (x) =\left( \Phi_{\varep, i} (x) \right)_{1\le i\le d}$ denotes the Dirichlet correctors for $\mathcal{L}_\varep$ in $\Omega$.

\begin{lemma}\label{lemma-5.1}
Let $J(\Phi_\varep)$ denote the absolute value of the determinant of the $d\times d$ matrix $\left(\frac{\partial \Phi_{\varep, i}}{\partial x_j}\right)$.
Then there exist constants $\varep_0>0$ and $c>0$,
depending only on $A$ and $\Omega$, such that for $0<\varep<\varep_0$,
$$
J(\Phi_\varep) (x) \ge c, \quad \text{ if }\ \  x\in \Omega \text{ and } \text{\rm dist}(x, \partial\Omega)\le c\,  \varep.
$$
\end{lemma}

\begin{proof}
Using dilation and the standard $C^{1, \alpha}$ estimate for $\mathcal{L}_1$, it is easy to see that
$$
|\nabla \Phi_\varep (x)-\nabla \Phi_\varep (y)|\le C \varep^{-\alpha} |x-y|^\alpha,
$$
for $x, y\in \Omega$ with $|x-y|\le \varep$, where $0<\alpha<1$ and $C$ depends only on $\alpha$, $A$, and $\Omega$.
This, together with the fact $\|\nabla \Phi_\varep\|_\infty\le C$, shows that it suffices to prove
$J(\Phi_\varep ) (x)\ge c>0$ for $x\in \partial\Omega$.

Next, we fix $P\in \partial\Omega$.
By translation and rotation we may assume that $P=0$ and
$$
\aligned
\Omega\cap \big\{ (x^\prime, x_d): \ & |x^\prime|<r_0 \text{ and } |x_d|<r_0 \big\}\\
&=\big\{ (x^\prime, x_d): \ |x^\prime|<r_0 \text{ and } \psi(x^\prime)<x_d<r_0 \big\},
\endaligned
$$
where $\psi:\br^{d-1} \to \br$ is a $C^2$ function such that $\psi (0)=|\nabla\psi (0)|=0$
and $\|\nabla^2 \psi\|_\infty\le M_0$.
Define
\begin{equation}\label{5.0.0}
U(r)=\big\{(x^\prime, x_d)\in \br^d: \ |x^\prime|<r \text{ and } \psi(x^\prime)<x_d <r\big\}.
\end{equation}
Since $\Phi_\varep (x)=x$ on $\partial\Omega$, we see that
$$
J(\Phi_\varep) (0)=\left|\frac{\partial\Phi_{\varep, d}}{\partial x_d} (0)\right|.
$$
Also recall that $ |\Phi_{\varep, d} (x) -x_d|\le C_0\, \varep$, where $C_0$ depends only on $A$.

Let $s_0> 4C_0$ be a large constant to be determined. For $0<\varep< (r_0/s_0)$,
let  $u_\varep$ be the solution of $\mathcal{L}_\varep (u_\varep)=0$ in $U(s_0\varep)$ with the Dirichlet data $g$
on $\partial U(s_0\varep)$, given by
\begin{equation}\label{5.0.1}
g(x)= \left\{\aligned
& M_0 |x^\prime|^2 &\quad & \text{ if } x_d =\psi (x^\prime) \text{ and } |x^\prime|<s_0\varep,\\
& M_0 (s_0\varep)^2 +C_0 \varep & \quad & \text{ if } |x^\prime|=s_0\varep \text{ and } \psi(x^\prime)<x_d<s_0\varep,\\
& 0 & \quad & \text{ if } x_d =s_0 \varep.
\endaligned
\right.
\end{equation}
Since $0\le g\le M_0 (s_0\varep)^2 +C_0\varep$, it follows from the maximum principle that 
$$
0\le u_\varep\le
M_0 (s_0\varep)^2 +C_0\varep \quad \text{ in } U(s_0\varep).
$$
By the boundary Lipschitz estimate in \cite[Lemma 20]{AL-1987}, we then obtain
\begin{equation}\label{5.0.2}
\aligned
|\nabla u_\varep (0)|  & \le 
C \left\{ s_0\varep+ (s_0\varep)^{-1} \max_{U(s_0\varep)} |u_\varep|  \right\}\\
& \le C_1 \left\{ s_0\varep +C_0 s_0^{-1} \right\},
\endaligned
\end{equation}
where $C_1$ depends only on $M_0$ and $A$.
Using $\Phi_{\varep, d} (x)\ge x_d-C_0\, \varep$ in $\Omega$ and $\Phi_{\varep, d} (x)=x_d$ on $\partial\Omega$,
it is easy to verify that $\Phi_{\varep, d} +g \ge 0$ on $\partial U(s_0\varep)$.
As a result, by the maximum principle, 
we also obtain $\Phi_{\varep, d} +u_\varep\ge 0$ on $U(s_0\varep)$.

Let $4C_0\le t_0<s_0$. We  consider the function
$$
w(x) =\Phi_{\varep,d} (t_0\varep x/2)  +u_\varep (t_0\varep x/2) \quad \text{ in } B =B (Q, 1),
$$
where $Q=(0, \dots, 0, 1)$ and $0< \varep< s_0^{-1} \min (r_0, (2 M_0)^{-1})$.
Note that 
$$
\mathcal{L}_{2t_0^{-1}} (w)=0 \quad \text{  in  }B \quad \text{ and }\quad 
\min_B w =w(0)=0.
$$
Thus, by the Hopf maximum principle (see e.g. \cite[p.330]{Evans}), we obtain
$$
\frac{\partial w}{\partial x_d} (0)\ge c_0 w (Q),
$$
where $c_0>0$ depends only on $t_0$ and $A$. It follows that
\begin{equation}\label{5.1.1}
\aligned
\frac{\partial \Phi_{\varep, d}}{\partial x_d} (0)  & \ge \frac{2c_0}{t_0\varep} 
\Phi_{\varep, d}  (0, \dots, 0, t_0\varep/2) -\frac{\partial u_\varep}{\partial x_d} (0)\\
&\ge 
\frac{2c_0}{t_0\varep} 
\Phi_{\varep, d}  (0, \dots, 0, t_0\varep/2)
-C_1 \big\{ s_0\varep +C_0 s_0^{-1} \big\},
\endaligned
\end{equation}
where we used the estimate (\ref{5.0.2}) as well as the fact $u_\varep\ge 0$.

Finally, note that if $t_0=4C_0$,
$$
\Phi_{\varep, d} (0, \dots, 0, t_0\varep/2)
\ge (t_0\varep/2) -C_0 \varep = (t_0\varep/4).
$$
 This, together with (\ref{5.1.1}) and the choice of $s_0=4C_1C_0/c_0$, yields
$$
\frac{\partial \Phi_{\varep, d}}{\partial x_d} (0)\ge \frac{c_0}{2}- C_1 s_0\varep -C_1 C_0 s_0^{-1}
\ge \frac{c_0}{8},
$$
for
$0<\varep<\varep_0$, where $\varep_0>0$ depends only on $A$ and $\Omega$.
The proof is complete.
\end{proof}

Since $\|\nabla \Phi_\varep\|_\infty\le C$,
it follows from Lemma \ref{lemma-5.1}
  that if $x\in \Omega$ and dist$(x, \partial\Omega)\le c\, \varep$, 
then the $d\times d$ matrix $(\nabla \Phi_\varep)$ is invertible at $x$ and
\begin{equation}\label{5.1-1}
c |w|\le  | (\nabla \Phi_\varep (x)) w |
\end{equation}
for any vector $w$ in $\br^d$.

\begin{lemma}\label{lemma-5.2}
Let $u_\varep$ be a Dirichlet eigenfunction for $\mathcal{L}_\varep$ in $\Omega$
with the associated eigenvalue $\lambda $ and $\| u_\varep \|_{L^2(\Omega)}
=1$. Then, if $0<\varep<\varep_0$,
\begin{equation}\label{estimate-5.2}
\frac{1}{\varep} \int_{\Omega_{c\varep}} |\nabla u_\varep|^2\, dx
\ge c\lambda -C\varep \lambda^2,
\end{equation}
where $c>0$ and $C>0$ depend only on $A$ and $\Omega$.
\end{lemma}

\begin{proof}
Let $v_\varep$ be the unique solution in $H^1_0(\Omega)$ to the equation
$\mathcal{L}_0 (v_\varep) =\lambda_\varep u_\varep$ in $\Omega$.
As in the proof of Theorem \ref{theorem-4.1}, we have
$\|\nabla v_\varep\|_{L^2(\Omega)} \le C \sqrt{\lambda}$ and $\|\nabla^2 v_\varep \|_{L^2(\Omega)}\le C \lambda$.
Moreover, it follows from (\ref{estimate-2.5})  that
\begin{equation}\label{5.2.1}
\|\nabla u_\varep -( \nabla \Phi_\varep)\nabla v_\varep \|_{L^2(\Omega)}
\le C \varep \lambda.
\end{equation}
Hence,
\begin{equation}
\aligned
\frac{1}{\varep}\int_{\Omega_{c\varep}} |\nabla u_\varep|^2\, dx
&\ge \frac{1}{2\varep} \int_{\Omega_{c\varep}} |(\nabla \Phi_{\varep}) \nabla v_\varep|^2\, dx
-C \varep \lambda^2\\
& \ge \frac{c}{\varep} \int_{\Omega_{c\varep}} |\nabla v_\varep|^2\, dx - C \varep \lambda^2,
\endaligned
\end{equation}
where we have used (\ref{5.1-1}) for the second inequality. Using
$$
\int_{\partial\Omega} |\nabla v_\varep|^2\, d\sigma
\le \frac{C}{\varep}\int_{\Omega_{c\varep}} |\nabla v_\varep|^2\, dx 
+ C\varep \int_{\Omega_{c\varep}} |\nabla^2 v_\varep |^2\, dx
$$
and $\|\nabla^2 v_\varep\|_{L^2 (\Omega)} \le C \lambda$, we further obtain
\begin{equation}\label{5.2.3}
\frac{1}{\varep}\int_{\Omega_{c\varep}} |\nabla u_\varep|^2\, dx
\ge c \int_{\partial\Omega} |\nabla v_\varep|^2\, d\sigma -C \varep\lambda^2.
\end{equation}
We will show that
\begin{equation}\label{5.2.4}
\lambda \le C\int_{\partial\Omega} |\nabla v_\varep|^2\, d\sigma + C \varep \lambda^2,
\end{equation}
which, together with (\ref{5.2.3}), yields the estimate (\ref{estimate-5.2}).

To see (\ref{5.2.4}), we may assume, without loss of generality, that $0\in \Omega$.
It follows by taking $h(x)=x$ in a Rellich identity for $\mathcal{L}_0$, similar to (\ref{Rellich-identity})
that
$$
\aligned
\int_{\partial\Omega} <x,n> \hat{a}_{ij} \frac{\partial v_\varep}{\partial x_j} \cdot \frac{\partial v_\varep}{\partial x_i}\, d\sigma
& =(2-d) \int_\Omega \hat{a}_{ij}\frac{\partial v_\varep}{\partial x_j }\cdot \frac{\partial v_\varep}{\partial x_i}\, dx
-2 \lambda \int_\Omega u_\varep \frac{\partial v_\varep}{\partial x_k} x_k \, dx\\
&=(2-d)\lambda \int_\Omega u_\varep v_\varep\, dx
-2 \lambda \int_\Omega u_\varep \frac{\partial v_\varep}{\partial x_k} x_k \, dx.
\endaligned
$$
This, together with 
$$
\aligned
2\int_\Omega u_\varep \frac{\partial v_\varep}{\partial x_k} x_k \, dx
& =-2\int_\Omega \frac{\partial u_\varep}{\partial x_k} v_\varep x_k \, dx
-2d\int_\Omega u_\varep v_\varep\, dx\\
&=d -2\int_\Omega \frac{\partial u_\varep}{\partial x_k} (v_\varep  -u_\varep) x_k\, dx
-2d \int_\Omega u_\varep v_\varep\, dx,
\endaligned
$$
obtained by integration by parts, gives
$$
\aligned
& \int_{\partial\Omega} <x,n> \hat{a}_{ij} \frac{\partial v_\varep}{\partial x_j} \cdot \frac{\partial v_\varep}{\partial x_i}\, d\sigma\\
& =2\lambda 
+(d+2)\lambda  \int_\Omega  u_\varep (v_\varep -u_\varep)\, dx
+2\lambda \int_\Omega \frac{\partial u_\varep}{\partial x_k}
(v_\varep -u_\varep) x_k \, dx.
\endaligned
$$
It follows that
\begin{equation}\label{5.2.6}
2\lambda   \le C \int_{\partial\Omega} |\nabla v_\varep|^2\, d\sigma
+ C \lambda \| u_\varep -v_\varep\|_{L^2(\Omega)}
+2 \lambda \left| \int_\Omega \frac{\partial u_\varep}{\partial x_k} (u_\varep -v_\varep ) x_k \, dx \right|.
\end{equation}

Finally, note that  $\| u_\varep -v_\varep\|_{L^2(\Omega)} \le C \varep \lambda$.
Also, the last term in the right hand side of (\ref{5.2.6}) is bounded
by
$$
\aligned
 &  2\lambda \left| \int_\Omega \frac{\partial u_\varep}{\partial x_k} \left[ u_\varep -v_\varep -(\Phi_{\varep, j} -x_j)
 \frac{\partial v_\varep}{\partial x_j} \right] x_k \, dx \right|
 + C \lambda\, \varep  \|\nabla u_\varep\|_{L^2(\Omega)}\| \nabla v_\varep\|_{L^2(\Omega)}\\
 &\le C\lambda\, \| \frac{\partial u_\varep}{\partial x_k} x_k \|_{H^{-1}(\Omega)}
 \| u_\varep -v_\varep -(\Phi_{\varep, j}-x_j) \frac{\partial v_\varep}{\partial x_j} \|_{H^1_0(\Omega)} + C\varep \lambda^2 \\
 &\le C \varep \lambda^2,
 \endaligned
 $$
 where we have used Theorem \ref{H-1-theorem}.
 This completes the proof of (\ref{5.2.4}).
 \end{proof}
 
 Let $\psi:\br^{d-1}\to \br$ be a $C^2$ function and $\psi (0)=|\nabla \psi (0)|=0$. Define
 $$
 \aligned
& Z_r =Z(\psi, r)  =\big\{ x=(x^\prime, x_d)\in \br^d: \, |x^\prime|<r \text{ and }  \psi(x^\prime)<x_d <r+\psi(x^\prime) \big\},\\
 & I_r = I(\psi, r) =\big\{ x=(x^\prime, x_d) \in \br^d: \, |x^\prime|< r \text{ and } x_d  =\psi(x^\prime) \big\}.
 \endaligned
 $$
 
 \begin{lemma}\label{lemma-5.3} Let $u\in H^1(Z_2)$.
 Suppose that $-\text{\rm div} (A\nabla u)  + E u =0$ in $Z_2$ and $u=0$ in $I_2$
 for some $E\in \br$.
Also assume that $|E| +\|\nabla A\|_\infty +\|\nabla^2\psi\|_\infty \le C_0$ and
\begin{equation}\label{5.3-1}
\int_{Z_1} |\nabla u|^2\, dx \ge c_0 \int_{Z_2} |\nabla u|^2\, dx
\end{equation}
for some $C_0>0, c_0>0$.
Then
\begin{equation}\label{5.3-2}
\int_{I_1} |\nabla u|^2\, d\sigma \ge c \int_{Z_2} |\nabla u|^2\, dx,
\end{equation}
 where $c>0$ depends only on the ellipticity constant $\kappa$ of $A$, $c_0$, and $C_0$.
  \end{lemma}

\begin{proof}
The lemma is proved by a compactness argument.
Suppose that there exist sequences $\{ \psi_k\}$ in $C^2(\br^{d-1})$, $\{ u_k\}$ in $H^1(Z(\psi_k, 2))$, $\{ E_k\}\subset \br$, and
 $\{ A^k(x) \}$ with ellipticity constant $\kappa$,  such that $\psi_k (0)=|\nabla \psi_k (0)|=0$,
\begin{equation}\label{5.3.1}
-\text{div}(A^k \nabla u_k) +E_k u_k =0  \text{ in } Z(\psi_k, 2), \quad 
\quad u_k =0 \text{ on } I (\psi_k, 2),
\end{equation}
\begin{equation}\label{5.3.2}
|E_k|+\|\nabla A^k\|_\infty + \|\nabla^2 \psi_k \|_\infty \le C_0,
\end{equation}
\begin{equation}\label{5.3.3}
\int_{Z(\psi_k, 2)} |\nabla u_k|^2 \, dx =1, \quad \int_{Z(\psi_k, 1)} |\nabla u_k|^2\, dx \ge c_0,
\end{equation}
and
\begin{equation}\label{5.3.5}
\int_{I (\psi_k, 1)} |\nabla u_k|^2\, d\sigma \to 0 \quad \text{ as } k\to \infty.
\end{equation}
By passing to a subsequence we may assume that $\psi_k \to \psi$ in $C^{1, \alpha}(|x^\prime|<4)$.
By the boundary $C^{1, \alpha}$ estimate we see that the norm of  $u_k$ in  $C^{1, \alpha} (Z(\psi_k, 3/2))$
is uniformly bounded.
As a result, by passing to a subsequence, we may assume that $v_k \to v$ in $C^1(Z(0,3/2))$,
where $v_k (x^\prime, x_d)=u_k (x^\prime, x_d -\psi_k (x^\prime))$ and
$Z(0, r)=\{ (x^\prime, x_d): \, |x^\prime|<r \text{ and } 0<x_d<r\}$.

We now let $u(x^\prime, x_d)=v(x^\prime, x_d +\psi(x^\prime))$.
Clearly, by passing to subsequences, we may also assume that
$E_k \to E$ in $\br$ and $A^k \to A$ in $C^\alpha (B(0, R_0))$.
It follows that $|E| +\|\nabla A\|_{L^\infty (B(0, R_0))}\le C_0$,
\begin{equation}\label{5.3.7}
-\text{div}(A\nabla u) +E u =0 \text{ in } Z(\psi, 1) \quad \text{ and } \quad u=0 \text{ on } I (\psi, 1).
\end{equation}
In view of (\ref{5.3.5}) we also obtain $\nabla u=0$ in $I (\psi, 1)$.
By the unique continuation property of solutions of second-order elliptic equations with Lipschitz continuous
coefficients (e.g. see \cite{AKS}), it follows that
$u=0$ in $Z(\psi, 1)$. However, by taking limit in  the inequality in  (\ref{5.3.3}), 
\begin{equation}\label{5.3.8}
\int_{Z(\psi, 1)} |\nabla u|^2\, dx \ge c_0>0.
\end{equation}
This gives us a contradiction and finishes the proof.
\end{proof}

\begin{remark}\label{remark-5.1}
{\rm
Suppose that $\mathcal{L}_\varep (u_\varep) =\lambda u_\varep$ in $Z(\psi, 2\varep)$
and $u_\varep =0$ in $I( \psi, 2\varep)$ for some $\lambda>1$.
Assume that $\varep^2 \lambda +\|\nabla A\|_\infty+\|\nabla^2 \psi\|_\infty\le C_0$ and
\begin{equation}\label{remark-5.1.1}
\int_{Z(\psi, \varep)} |\nabla u_\varep|^2 \, dx
\ge c_0 \int_{Z(\psi, 2\varep)} |\nabla u_\varep|^2\, dx
\end{equation}
for some $c_0, C_0>0$. Then
\begin{equation}\label{remark-5.1.3}
\int_{I (\psi, \varep)}
|\nabla u_\varep|^2\, d\sigma 
\ge \frac{c}{\varep} \int_{Z(\psi, 2\varep)} |\nabla u_\varep|^2\, dx,
\end{equation}
where $c>0$ depends only on the ellipticity constant of $A$, $c_0$, and $C_0$.
This is a simple consequence of Lemma \ref{lemma-5.3}. Indeed,
let $w(x)= u_\varep (\varep x)$ and $\psi_\varep (x^\prime)=\varep^{-1}\psi (\varep x^\prime)$.  Then 
$\mathcal{L}_1 (w)=\varep^2\lambda w$ in $Z(\psi_\varep, 2)$ and
$$
\int_{Z(\psi_\varep, 1)} |\nabla w|^2\, dx \ge c_0 \int_{Z(\psi_\varep, 2)} |\nabla w|^2\, dx.
$$
Since $\varep^2 \lambda +\|\nabla A\|_\infty +\|\nabla^2 \psi_\varep\|_\infty\le C_0$, it follows from 
Lemma \ref{lemma-5.3} that
\begin{equation}\label{remark-5.1.5}
\int_{I(\psi_\varep, 1)} |\nabla w|^2\, d\sigma
\ge c \int_{Z(\psi_\varep, 2)} |\nabla w|^2\, dx,
\end{equation}
which gives (\ref{remark-5.1.3}).
Note that the periodicity assumption of $A$ is not needed here.
}
\end{remark}

\begin{proof}[\bf Proof of Theorem \ref{theorem-C}]
For each $P\in \partial\Omega$,  there exists a  new coordinate system of $\br^d$,
 obtained from the standard Euclidean coordinate system through
translation and rotation, so that $P=(0, 0)$ and
$$
\Omega \cap B(P, r_0) =\big\{ (x^\prime, x_d)\in \br^d: \ x_d>\psi (x^\prime)\big\} \cap B(P, r_0),
$$
where $\psi (0)=|\nabla \psi (0)| =0$ and $\|\nabla^2 \psi\|_\infty \le M$.
For $0<r<cr_0$, let $(\Delta (P, r), D(P, r))$ denote the pair obtained from $(I(\psi, r), Z(\psi, r))$ by this change of 
the coordinate system. 
If $0<\varep<c r_0$, we may construct a finite sequence of pairs $\{ \big(\Delta (P_i, \varep), D_i (P_i, \varep)\big)\}$ such that
$$
\partial\Omega =\bigcup_i \Delta (P_i, \varep)
$$
and
\begin{equation}\label{5.5.0}
\sum_i \chi_{{D} (P_i, 2\varep)} \le C \quad \text{ and } \quad \Omega_{c\varep} \subset \bigcup_i D(P_i, \varep).
\end{equation}
Let $\Delta_i (r)=\Delta (P_i, r)$ and $D_i (r)=D_i (P, r)$.

Suppose now that $u_\varep \in H^1_0(\Omega)$, $\mathcal{L}_\varep (u_\varep ) =\lambda u_\varep$ in $\Omega$,
and $\| u_\varep\|_{L^2(\Omega)}=1$.
Assume that $\lambda>1$ and $\varep \lambda\le \delta$, where $\delta=\delta (A, \Omega)>0$ is sufficiently small.
It follows from Lemma \ref{lemma-5.2} and (\ref{4.5-5})-(\ref{4.5-6}) in the proof of Theorem \ref{theorem-4.1}  that
\begin{equation}\label{5.5.1}
c\lambda \le \frac{1}{\varep} \int_{\Omega_{c\varep} }  |\nabla u_\varep |^2\, dx
\le \frac{1}{\varep} \int_{\Omega_{2\varep}} |\nabla u_\varep|^2\, dx \le C \lambda.
\end{equation}
To estimate $\int_{\partial\Omega} |\nabla u_\varep|^2\, d\sigma$ from below, we divide $\{ D_i (\varep)\}$
into two groups. We call $i \in J$ if
\begin{equation}\label{5.5.2}
\int_{{D}_i(2\varep)} |\nabla u_\varep|^2\, dx
\le N \int_{D_i (\varep)} |\nabla u_\varep|^2\, dx
\end{equation}
with a large constant $N=N(A, \Omega)$ to be determined. Note that if $i \in J$,
by Remark \ref{remark-5.1}, 
$$
\int_{\Delta_i (\varep) } |\nabla u_\varep|^2\, d\sigma 
\ge \frac{\gamma }{\varep} \int_{D_i (\varep)} |\nabla u_\varep|^2\, dx,
$$
where $\gamma >0$ depends only on $A$, $\Omega$, and $N$. It follows by summation that
\begin{equation}\label{5.5.3}
\aligned
\int_{\partial\Omega}
|\nabla u_\varep|^2\, d\sigma & \ge 
\frac{c\gamma }{\varep} \int_{\cup_{i\in J} {D}_i (\varep)} |\nabla u_\varep|^2\, dx\\
&\ge \frac{c\gamma }{\varep} \left\{ \int_{\Omega_{c\varep}} |\nabla u_\varep|^2\, dx
-
 \int_{\cup_{i\notin J} {D}_i (\varep)} |\nabla u_\varep|^2\, dx\right\} \\
&\ge \frac{c\gamma}{\varep}
\left\{ c\varep \lambda -  \int_{\cup_{i\notin J} {D}_i (\varep)} |\nabla u_\varep|^2\, dx\right\},
\endaligned
\end{equation}
where we have used the fact $\Omega_{c\varep} \subset \bigcup_i {D}_i (\varep)$ and estimate  (\ref{5.5.1}).

Finally, we note that by the definition of $J$ as well as the estimate (\ref{5.5.1}),
$$
 \int_{\cup_{i\notin J} {D}_i (\varep)} |\nabla u_\varep|^2\, dx
\le 
\frac{C}{N} \int_{\cup_{i\notin J} D_i(2\varep)} |\nabla u_\varep|^2\, dx
\le \frac{C}{N} \int_{\Omega_{2\varep}} |\nabla u_\varep|^2\, dx
\le \frac{C\varep \lambda}{N},
$$
where we have used the fact $\bigcup_i D_i(2\varep) \subset \Omega_{2\varep}$.
This, together with (\ref{5.5.3}), yields
\begin{equation}\label{5.5.5}
\int_{\partial\Omega} |\nabla u_\varep|^2\, d\sigma
\ge c\gamma \lambda \left\{ c-CN^{-1}\right\}
\ge c\lambda,
\end{equation}
if $N =N(A, \Omega)$ is sufficiently large. The proof is complete.
\end{proof}

\bibliography{kls4}

\small
\noindent\textsc{Department of Mathematics, 
University of Chicago, Chicago, IL 60637}\\
\emph{E-mail address}: \texttt{cek@math.uchicago.edu} \\

\noindent \textsc{Courant Institute of Mathematical Sciences, New York University, New York, NY 10012}\\
\emph{E-mail address}: \texttt{linf@cims.nyu.edu}\\

\noindent\textsc{Department of Mathematics, 
University of Kentucky, Lexington, KY 40506}\\
\emph{E-mail address}: \texttt{zshen2@uky.edu} \\

\noindent \today

\end{document}